\documentclass[a4paper,11pt]{article}%
\pdfoutput=1
\usepackage[T1]{fontenc}
\usepackage{amsfonts}
\usepackage{amssymb}
\usepackage{textcomp}
\usepackage{typearea}
\usepackage{enumerate}
\usepackage{amsmath}
\usepackage{amsthm}
\usepackage{a4}
\usepackage{vmargin}
\usepackage[center]{subfigure}
\usepackage{todonotes}
\usepackage{scrpage2}
\usepackage[font=small,format=plain,labelfont=bf,up,textfont=it,up]{caption}
\usepackage{enumitem}
\usepackage{graphicx}%
\providecommand{\U}[1]{\protect\rule{.1in}{.1in}}
\newtheorem{theorem}{Theorem}

\newtheorem{remark}[theorem]{Remark}

\graphicspath{{Pictures/}}

\newcommand{\vertiii}[1]{{\left\vert\kern-0.25ex\left\vert\kern-0.25ex\left\vert#1
\right\vert\kern-0.25ex\right\vert\kern-0.25ex\right\vert}}

\numberwithin{equation}{section}
\numberwithin{theorem}{section}

\numberwithin{figure}{section}
\newcommand{\beq}{\begin{equation}}
\newcommand{\eeq}{\end{equation}}

\begin{document}

\title{Adaptive Time Discretization for Retarded Potentials}
\author{S. Sauter\thanks{Institut f\"{u}r Mathematik, Universit\"{a}t Z\"{u}rich,
Winterthurerstrasse 190, 8057 Z\"{u}rich, Switzerland, e-mail:
\texttt{stas@math.uzh.ch}}
\and A. Veit\thanks{Department of Computer Science, University of Chicago, Chicago, Illinois 60637, USA, e-mail:
\texttt{aveit@uchicago.edu}} }
\date{}
\maketitle

\begin{abstract}
In this paper, we will present advanced discretization methods for solving
retarded potential integral equations. We employ a $C^{\infty}$-partition of
unity method in time and a conventional boundary element method for the
spatial discretization. One essential point for the algorithmic realization is
the development of an efficient method for approximation the elements of the
arising system matrix. We present here an approach which is based on
quadrature for (non-analytic) $C^{\infty}$ functions in combination with
certain Chebyshev expansions.

Furthermore we introduce an a posteriori error estimator for the time
discretization which is employed also as an error indicator for adaptive
refinement. Numerical experiments show the fast convergence of the proposed
quadrature method and the efficiency of the adaptive solution process.\medskip

\medskip


\end{abstract}

\section{Introduction}

In this paper, we will consider the efficient numerical solution of the wave
equation in unbounded domains. The exact solution is represented as a retarded
potential and the arising space-time boundary integral equation (RPIE) is
solved numerically by using a Galerkin method in time and space
(\cite{friedman}, \cite{Bam1}, \cite{HaDuong}).

The novelties compared to existing methods (\cite{Bam1}, \cite{Ding},
\cite{HaDuong}, \cite{HaDuong2}, \cite{Ned4}, \cite{Sauter5}, \cite{Weile3})
are as follows.

\begin{itemize}
\item[a)] We employ a $C^{\infty}$-partition of unity enriched by polynomials for the
temporal discretization as introduced in \cite{Sauter5}. This approach
overcomes the technical difficulty to first determine and then to integrate
over the intersection of the discrete light cone with the spatial mesh which
arises if conventional piecewise polynomial finite elements are employed in
time (cf. \cite{HaDuong2}). However, the arising quadrature problem for our
$C^{\infty}$ basis functions is not completely standard since the functions
are not analytic. In this paper we will propose an efficient method to
approximate the arising integrals and perform systematic numerical experiments
to demonstrate its fast convergence. It turns out that for the important range
of accuracies $\left[  10^{-1},10^{-8}\right]  $ the method converges nearly
as fast as for analytic integrands.

\item[b)] We present an a posteriori error estimator for retarded potential integral
equations which also is employed as a refinement indicator for an adaptive solution
process. To the best of our knowledge this is the first time that an
self-adaptive method is proposed for RPIE in 3D (for the 2D case we refer to the thesis \cite{Glaefke}; for adaptive versions of
the convolution quadrature method we refer to \cite{LopezSauter2011} and \cite{lopezsauter_paper3}). The error estimator is based on the
estimator which was proposed in \cite{Faermann}, \cite{FAERMANN3D} for
elliptic boundary integral equation. We will present numerical experiments
where the solution contains sharp pulses and/or oscillations at different time
scales and time windows. Our error indicator captures very well the
irregularities in the solution and marks for refinement at the
\textquotedblleft right\textquotedblright\ places. These experiments also
indicate that a global error estimator in time is essential for setting up an
adaptive method since it seems to be quite complicated for a \textit{time
stepping} scheme to detect the regions in the time history which causes the
error at the current time step.
\begin{remark}\label{rem_longTerm} We emphasize that the long term goal of this research
is to develop a space-time a posteriori error estimator and the
resulting algorithm should be fully space-time adpative.
In this paper we will present a purely temporal a posteriori error 
estimator. It turns out that
this algorithm is able to capture local irregularities with respect to 
time very well.
We expect that a generalization of this estimator to a
space-time adaptive method allows to reduce the dimensions of
spatial boundary element matrices substantially so that
the loss of the Toeplitz structure in the linear system
becomes negligible due to the much smaller dimension of the
full system matrix.
In any case, a reliable a posteriori error estimator is important
also for uniform mesh refinement and serves as a computable
upper bound for the error which can be used as a stopping criterion.
\end{remark}

\item[c)] We present systematic numerical experiments to understand i) the
convergence behavior of the spatial quadrature depending on the distance of
the pairs of panels and the width of the discrete light cone, ii) the
influence of the spatial quadrature to the overall discretization error as
well as the convergence rates with respect to the energy norm, iii) the long
term stability behavior of our space-time Galerkin approach also in
comparison with the convolution quadrature method (\cite{Lub1}), iv) the
performance of the new self-adaptive method which is based on our a posteriori
error estimator.
\end{itemize}
The paper is structured as follows. After the retarded potential integral
equation will be introduced in Section \ref{Sec2} we explain its numerical
discretization in Section \ref{SecNumDisc} as well as the numerical
approximation of the entries of the system matrix. In Section \ref{Sec:ErrEst}%
, the a posteriori error estimator is formulated and its numerical evaluation
is explained. Numerical experiments are presented in Section \ref{NumEx} which
give insights in the performance of the various discretization methods and
their influence to the overall discretization. The method and its main
features are summarized in the concluding Section \ref{SecConcl}.

\section{Integral Formulation of the Wave Equation\label{Sec2}}

Let $\Omega\subset\mathbb{R}^{3}$ be a Lipschitz domain with boundary $\Gamma
$. We consider the homogeneous wave equation%
\begin{subequations}
\label{FullProblem}
\end{subequations}%
%
\begin{equation}
\partial_{t}^{2}u-\Delta u=0\hspace{3mm}\text{in }\Omega\times\left[
0,T\right]  \tag{%
\ref{FullProblem}%
a}\label{WaveEquation}%
\end{equation}
with initial conditions%
\begin{equation}
u(\cdot,0)=\partial_{t}u(\cdot,0)=0\hspace{3mm}\text{in }\Omega\tag{%
\ref{FullProblem}%
b}\label{InitialConditions}%
\end{equation}
and Dirichlet boundary conditions%
\begin{equation}
u=g\hspace{3mm}\text{on }\Gamma\times\left[  0,T\right]  \tag{%
\ref{FullProblem}%
c}\label{BoundaryConditions}%
\end{equation}
on a time interval $\left[  0,T\right]  $ for $T>0$. In applications, $\Omega$
is often the unbounded exterior of a bounded domain. For such problems, the
method of boundary integral equations is an elegant tool where this partial
differential equation is transformed to an equation on the bounded surface
$\Gamma$. We employ an ansatz as a \textit{single layer potential} for the
solution $u$,
\begin{equation}
u(x,t):=S\phi(x,t):=\int_{\Gamma}\frac{\phi(y,t-\Vert x-y\Vert)}{4\pi\Vert
x-y\Vert}d\Gamma_{y},\hspace{3mm}\ (x,t)\in\Omega\times\left[  0,T\right]
\label{RetardedPotential}%
\end{equation}
with unknown density function $\phi$. $S$ is also referred to as
\textit{retarded single layer potential} due to the retarded time argument
$t-\Vert x-y\Vert$ which connects time and space variables.\vspace{0pt}

The ansatz (\ref{RetardedPotential}) satisfies the wave equation
(\ref{WaveEquation}) and the initial conditions (\ref{InitialConditions}).
Since the single layer potential can be extended continuously to the boundary
$\Gamma$, the unknown density function $\phi$ is determined such that the
boundary conditions (\ref{BoundaryConditions}) are satisfied. This results in
the boundary integral equation for $\phi$,
\begin{equation}
\int_{\Gamma}\frac{\phi(y,t-\Vert x-y\Vert)}{4\pi\Vert x-y\Vert}d\Gamma
_{y}=g(x,t)\hspace{3mm}\forall(x,t)\in\Gamma\times\left[  0,T\right]
.\label{BIE}%
\end{equation}
In order to solve this boundary integral equation numerically we introduce a
weak formulation of \eqref{BIE} according to \cite{Bam1, HaDuong}. Therefore
we introduce the space
\begin{align*}
H^{-1/2,-1/2}(\Gamma\times\lbrack0,T])&:=L^{2}([0,T],H^{-1/2}(\Gamma
))+H^{-1/2}([0,T],L^{2}(\Gamma)) .
\end{align*}
A suitable space-time variational formulation of \eqref{BIE} is then given by:
Find $\phi\in H^{-1/2,-1/2}(\Gamma\times\lbrack0,T])$ s.t.
\begin{align}
a(\phi,\zeta):=\int_{0}^{T}\int_{\Gamma}\int_{\Gamma}\frac{\dot{\phi
}(y,t-\Vert x-y\Vert)\zeta(x,t)}{4\pi\Vert x-y\Vert} &  d\Gamma_{y}d\Gamma
_{x}dt\nonumber\\
&  =\int_{0}^{T}\int_{\Gamma}\dot{g}(x,t)\zeta(x,t)d\Gamma_{x}dt=:b(\zeta
)\label{VarForm}%
\end{align}
for all $\zeta\in H^{-1/2,-1/2}(\Gamma\times\lbrack0,T])$, where we denote by
$\dot{\phi}$ the derivative with respect to time. It can be shown that
$a(\cdot,\cdot)$ is coercive in $H^{-1/2,-1/2}(\Gamma\times\lbrack0,T])$,
i.e.
\begin{equation}
a(\phi,\phi)\geq C\Vert\phi\Vert_{H^{-1/2,-1/2}(\Gamma\times\lbrack0,T])}%
^{2}.\label{coercivity}%
\end{equation}
This, together with an energy argument, can be used to show unconditional
stability of conforming Galerkin approximations (cf. \cite{Bam1, HaDuong}) of \eqref{VarForm}.

\section{Numerical Discretization\label{SecNumDisc}}

We discretize the variational problem \eqref{VarForm} using a Galerkin method
in space and time. Therefore we replace $H^{-1/2,-1/2}(\Gamma\times
\lbrack0,T])$ by a finite dimensional subspace $V_{\operatorname{Galerkin}}$
being spanned by some basis functions $\{b_{i}\}_{i=1}^{L}$ in time and some
basis functions $\{\varphi_{j}\}_{j=1}^{M}$ in space. This leads to the
discrete ansatz
\begin{equation}
\phi_{\operatorname{Galerkin}}(x,t)=\sum_{i=1}^{L}\sum_{j=1}^{M}\alpha_{i}%
^{j}\varphi_{j}(x)b_{i}(t),\hspace{3mm}(x,t)\in\Gamma\times\left[  0,T\right]
\label{DiscreteAnsatz}%
\end{equation}
for the approximate solution, where $\alpha_{i}^{j}$ are the unknown
coefficients. Plugging \eqref{DiscreteAnsatz} into the variational formulation
\eqref{VarForm} and using the basis functions $b_{k}$ and $\varphi_{l}$ as
test functions leads to the linear system
\[
\underline{\underline{\mathbf{A}}}\cdot\underline{\boldsymbol{\alpha}%
}=\underline{\mathbf{g}},
\]
where the block matrix $\underline{\underline{\mathbf{A}}}\in\mathbb{R}%
^{LM\times LM}$, the unknown coefficient vector $\underline{\boldsymbol{\alpha
}}\in\mathbb{R}^{LM}$ and the right-hand side vector $\underline{\mathbf{g}%
}\in\mathbb{R}^{LM}$ can be partitioned according to
\begin{equation}
\underline{\underline{\mathbf{A}}}:=%
\begin{bmatrix}
\mathbf{A}_{1,1} & \mathbf{A}_{1,2} & \cdots & \mathbf{A}_{1,L}\\
\mathbf{A}_{2,1} & \mathbf{A}_{2,2} & \cdots & \mathbf{A}_{2,L}\\
\vdots & \vdots & \ddots & \vdots\\
\mathbf{A}_{L,1} & \mathbf{A}_{L,2} & \cdots & \mathbf{A}_{L,L}%
\end{bmatrix}
,\qquad\underline{\boldsymbol{\alpha}}:=%
\begin{bmatrix}
\boldsymbol{\alpha}_{1}\\
\boldsymbol{\alpha}_{2}\\
\vdots\\
\boldsymbol{\alpha}_{L}%
\end{bmatrix}
,\qquad\underline{\mathbf{g}}:=%
\begin{bmatrix}
\mathbf{g}_{1}\\
\mathbf{g}_{2}\\
\vdots\\
\mathbf{g}_{L}%
\end{bmatrix}
,\label{A}%
\end{equation}
with
\[
\mathbf{A}_{k,i}\in\mathbb{R}^{M\times M},\quad\boldsymbol{\alpha}_{i}%
\in\mathbb{R}^{M},\quad\mathbf{g}_{k}\in\mathbb{R}^{M}\quad\text{ for }%
i,k\in\{1,\cdots,L\}.
\]
Their entries are given by
\begin{equation}
\mathbf{A}_{k,i}(j,l)=\int_{0}^{T}\int_{\Gamma}\int_{\Gamma}\frac{\varphi
_{j}(y)\,\varphi_{l}(x)}{4\pi\Vert x-y\Vert}\,\dot{b}_{i}(t-\Vert
x-y\Vert)b_{k}(t)\,d\Gamma_{y}d\Gamma_{x}dt\label{integrals}%
\end{equation}
and
\[
\boldsymbol{\alpha}_{i}(j)=\left(  \alpha_{i}^{j}\right)  _{j=1}^{M}%
,\quad\mathbf{g}_{k}(l)=\int_{0}^{T}\int_{\Gamma}\dot{g}(x,t)\,\varphi
_{l}(x)\,b_{k}(t)d\Gamma_{x}dt
\]
respectively. We rewrite \eqref{integrals} by introducing a univariate
function $\psi_{i,k}$ with
\begin{equation}
\psi_{k,i}(r)=\int_{0}^{T}\dot{b}_{i}(t-r)b_{k}(t)dt
\end{equation}
and obtain
\begin{align}
\mathbf{A}_{k,i}(j,l) &  =\int_{\Gamma}\int_{\Gamma}\frac{\varphi
_{j}(y)\,\varphi_{l}(x)}{4\pi\Vert x-y\Vert}\psi_{k,i}(\Vert x-y\Vert
)\,d\Gamma_{y}d\Gamma_{x}\nonumber\\
&  =\int_{\text{supp}(\varphi_{l})}\int_{\text{supp}(\varphi_{j})}%
\frac{\varphi_{j}(y)\,\varphi_{l}(x)}{4\pi\Vert x-y\Vert}\psi_{k,i}(\Vert
x-y\Vert)\,d\Gamma_{y}d\Gamma_{x}.\label{integrals2}%
\end{align}
The efficient and accurate computation of the matrix entries \eqref{integrals2}
is crucial for this method and represents a major challenge in the space-time
Galerkin approach. The choice of the basis functions in time plays here a
significant role.
In this paper we use smooth and compactly supported temporal shape functions
$b_{i}$ in \eqref{DiscreteAnsatz} whose definition was addressed in
\cite{Sauter5}. For the sake of a self-contained presentation we briefly
recall their definition. Let%
\[
f\left(  t\right)  :=\left\{
\begin{array}
[c]{ll}%
\frac{1}{2}\operatorname{erf}\left(  2\operatorname{artanh}t\right)  +\frac
{1}{2} & \left\vert t\right\vert <1,\\
0 & t\leq-1,\\
1 & t\geq1
\end{array}
\right.
\]
and note that $f\in C^{\infty}\left(  \mathbb{R}\right)  $. Next, we will
introduce some scaling. For a function $g\in C^{0}\left(  \left[  -1,1\right]
\right)  $ and real numbers $a<b$, we define $g_{a,b}\in C^{0}\left(  \left[
a,b\right]  \right)  $ by%
\[
g_{a,b}\left(  t\right)  :=g\left(  2\frac{t-a}{b-a}-1\right)  .
\]
We obtain a bump function on the interval $\left[  a,c\right]  $ with joint
$b\in\left(  a,c\right)  $ by%
\[
\rho_{a,b,c}\left(  t\right)  :=\left\{
\begin{array}
[c]{ll}%
f_{a,b}\left(  t\right)   & a\leq t\leq b,\\
1-f_{b,c}\left(  t\right)   & b\leq t\leq c,\\
0 & \text{otherwise.}%
\end{array}
\right.
\]
Let us now consider the closed interval $\left[  0,T\right]  $ and $l$ (not
necessarily equidistant) timesteps
\begin{equation}
0=t_{0}<t_{1}<\ldots t_{l-2}<t_{l-1}=T.\label{timegrid}%
\end{equation}
A smooth partition of unity of the interval $[0,T]$ then is defined by%
\[
\mu_{1}:=1-f_{t_{0},t_{1}},\quad\mu_{l}:=f_{t_{l-2,l-1}},\quad\forall2\leq
i\leq l-1:\mu_{i}:=\rho_{t_{i-2},t_{i-1},t_{i}}.
\]
Smooth and compactly supported basis functions $b_{i}$ in time can then be
obtained by multiplying these partition of unity functions with suitably
scaled Legendre polynomials (cf. \cite{Sauter5} for details):
\begin{alignat}{2}
&  \mu_{1}(t)\cdot8\cdot\left(  \frac{t}{t_{1}}\right)  ^{2}P_{m-2}\left(
\frac{2}{t_{1}}t-1\right)  \quad m=2,\ldots,\max(2,p), & \nonumber\\
&  \mu_{i}(t)P_{m}\left(  2\frac{t-t_{i-2}}{t_{i}-t_{i-2}}-1\right)\quad
m=0,\ldots,p,\hspace{1mm}i=2,\ldots,l-1, & \label{BasisFunctions}\\
&  \mu_{l}(t)P_{m}\left(  2\frac{t-t_{l-2}}{t_{l-1}-t_{l-2}}-1\right)  \quad
m=0,\ldots,p. & \nonumber
\end{alignat}
We will use the above basis functions in time for the Galerkin approximation
in \eqref{DiscreteAnsatz}. The order of the approximation in time can be
controlled by $p$ in \eqref{BasisFunctions}. For the choice $p=0$ the solution
is approximated in time merely with the partition of unity functions $\mu_{i}%
$. This corresponds to the approximation with piecewise constant functions in
the standard Galerkin approach.\newline For the discretization in space we use
standard piecewise polynomials basis functions $\varphi_{j}$.


\subsection{Efficient evaluation of $\psi_{k,i}$}

The approximation of the matrix entries using quadrature is the most time
consuming part of the method. In order to reduce the computational time, an
efficient evaluation of the integrand in \eqref{integrals2} is crucial. Since
$\psi_{k,i}$ consists itself of an integral this evaluation can typically not
be done exactly and has to be approximated. One obvious strategy is to apply
Gauss-Legendre quadrature also to the integral in $\psi_{k,i}$. In order to
obtain accurate results this unfortunately requires a relatively high number
of quadrature nodes and furthermore the basis functions $\dot{b_{i}}$ and
$b_{k}$ have to be evaluated multiple times which is itself expensive due to
the presence of the error function and the inverse hyperbolic tangent.\newline
In order to speed up the evaluation of \eqref{integrals2} we therefore want to
represent $\psi_{k,i}$ accurately by functions that are easy to construct and
allow a fast evaluation. Since $\psi_{k,i}$ is smooth and compactly supported
we choose piecewise Chebyshev polynomials for this task. We introduce%
\[
\text{min}_{k}:=\min\operatorname{supp}b_{k}\quad\text{ and }\quad
\text{max}_{k}:=\max\operatorname{supp}b_{k}%
\]
for all $1\leq k\leq L$, so that
\[
\operatorname{supp}\psi_{k,i}=[\text{min}_{k}-\text{max}_{i},\text{max}%
_{k}-\text{min}_{i}]=:[a,b].
\]
We divide $[a,b]$ into $m$ subintervals of length
\[
h_{m}:=\frac{b-a}{m}%
\]
and denote
\[
\Delta_{m,j}:=[a+(j-1)h_{m},a+jh_{m}]
\]
for $j=1,\ldots,m$. We approximate $\psi_{k,i}$ on each subinterval by a
linear combination of Chebyshev polynomials $T_{v}$ of degree $v$, i.e.,
\begin{equation}
\psi_{k,i}(r)|_{\Delta_{m,j}}\approx\sum_{v=0}^{q-1}c_{v}T_{v}(\varphi
(r))-\frac{1}{2}c_{0},\label{Chebapprox}%
\end{equation}
where
\[
\varphi:\Delta_{m,j}\rightarrow\lbrack-1,1],\quad r\mapsto\frac{2r-(\max
\Delta_{m,j}+\min\Delta_{m,j})}{\max\Delta_{m,j}-\min\Delta_{m,j}}%
\]
is an appropriate scaling function. The coefficients $c_{v}$ are defined by
\[
c_{v}=\frac{2}{q}\sum_{k=1}^{q}\psi_{k,i}\left(  \varphi^{-1}\left[
\cos\left(  \frac{\pi(k-0.5)}{q}\right)  \right]  \right)  \cos\left(
\frac{\pi v(k-0.5)}{q}\right)  \qquad0\leq v\leq q-1
\]
which can be evaluated efficiently using fast cosine transform methods. The
evaluation of the Chebyshev approximation \eqref{Chebapprox} can be done with
Clenshaw's recurrence formula (cf. \cite[Chapter 5.5]{Press92}).

\begin{remark}
The approximation of $\psi_{k,i}$ using the piecewise polynomials
\eqref{Chebapprox} requires the evaluation of $\psi_{k,i}$ at $q\cdot m$
different points. Note that this has to be done only once for each matrix
block $\mathbf{A}_{k,i}$. In order to obtain accurate results we therefore use
high-order Gauss-Legendre quadrature for the evaluation of $\psi_{k,i}$ at
these points.
\end{remark}

Numerical experiments indicate that the accuracy of the approximation in
\eqref{Chebapprox} has a significant impact on the accuracy of the
approximation of \eqref{integrals2} using Gauss-Legendre quadrature. The
number of subintervals $m$ and the polynomial degree $q$ of the piecewise
approximations \eqref{Chebapprox} should therefore be chosen such that the
error $\Vert\psi_{k,i}-\psi_{k,i}^{\text{approx}}\Vert_{\infty}$ is
sufficiently small; in our numerical
experiments a threshold of $10^{-8}$ for this error always preserved
the asymptotic convergence rates. We have performed numerical experiments to
assemble a table with optimal pairs $\left(  m,q\right)  $ for certain
accuracies. As model situations we have considered the (nonuniform) time grid
\[
t_{0}=0,\quad t_{1}=2,\quad t_{2}=3,\quad t_{3}=4.5,\quad t_{4}=7
\]
and chosen bump functions $\rho_{t_{0},t_{1},t_{2}}$ and $\rho_{t_{2}%
,t_{3},t_{4}}$ as above. Let%
\begin{alignat}{2}
&  b_{1}(t):=\rho_{t_{0},t_{1},t_{2}}(t), \quad && b_{2}(t):=\rho_{t_{2},t_{3},t_{4}%
}(t),  \nonumber\\
&  b_{3}(t):=\rho_{t_{0},t_{1},t_{2}}(t)P_{3}\left(  2\frac{t-t_{0}}%
{t_{2}-t_{0}}-1\right)  , \quad  && b_{4}(t):=\rho_{t_{2},t_{3},t_{4}}%
(t)P_{2}\left(  2\frac{t-t_{2}}{t_{4}-t_{2}}-1\right)    \nonumber
\end{alignat}
be functions of the type \eqref{BasisFunctions}. Next, we define
\[
\psi_{1}:\mathbb{R}\rightarrow\mathbb{R},r\mapsto\int_{0}^{7}\dot{b}%
_{1}(t-r)b_{2}(t)dt\quad\text{and}\quad\psi_{2}:\mathbb{R}\rightarrow
\mathbb{R},r\mapsto\int_{0}^{7}\dot{b}_{3}(t-r)b_{4}(t)dt.
\]

\begin{figure}[ptbh]
\centering
\begin{minipage}[t]{0.49\textwidth}
\includegraphics[width=1\textwidth]{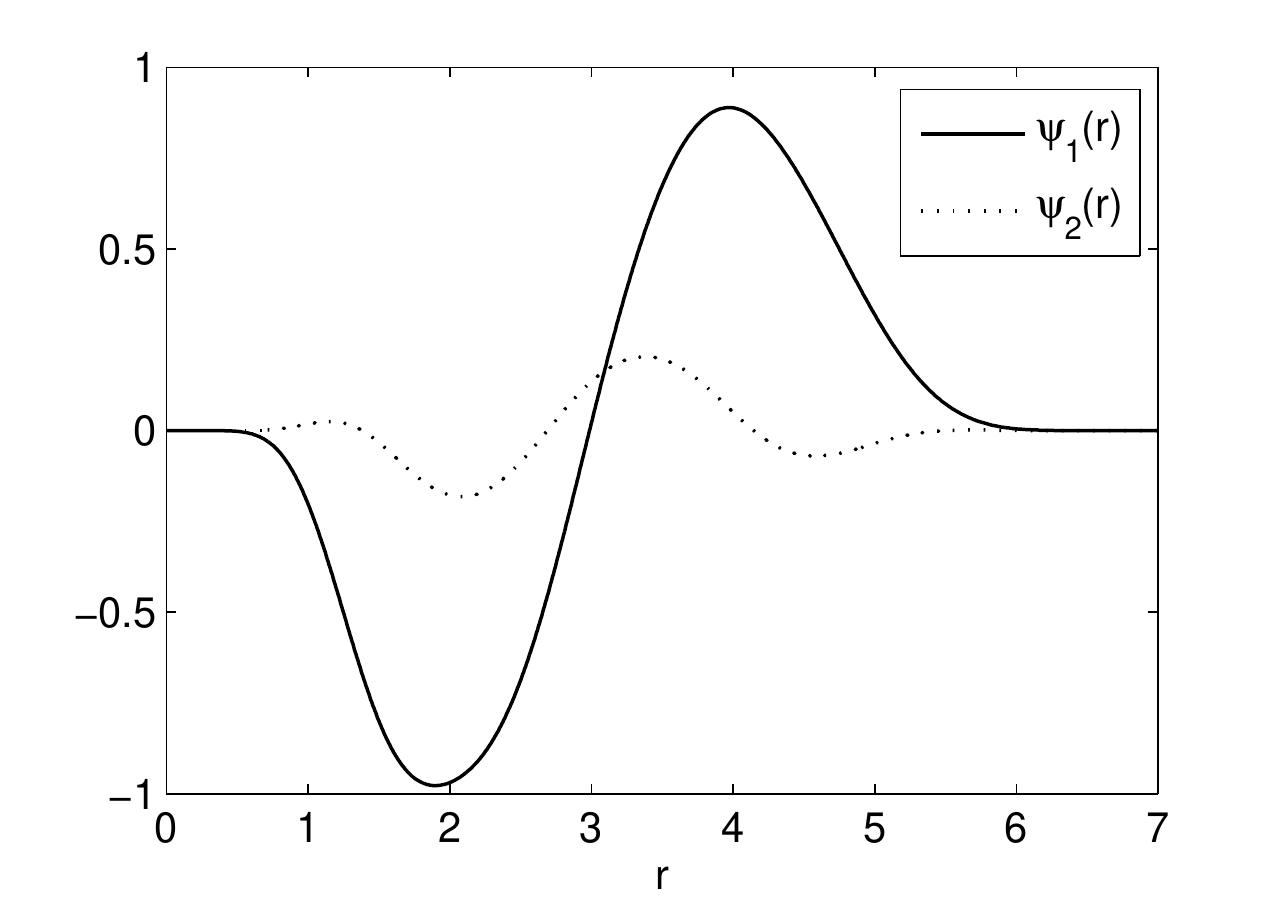}
\caption{$\psi_1(r)$ and $\psi_2(r)$}
\label{Fig:psi}
\end{minipage}
\begin{minipage}[t]{0.49\textwidth}
\includegraphics[width=1\textwidth]{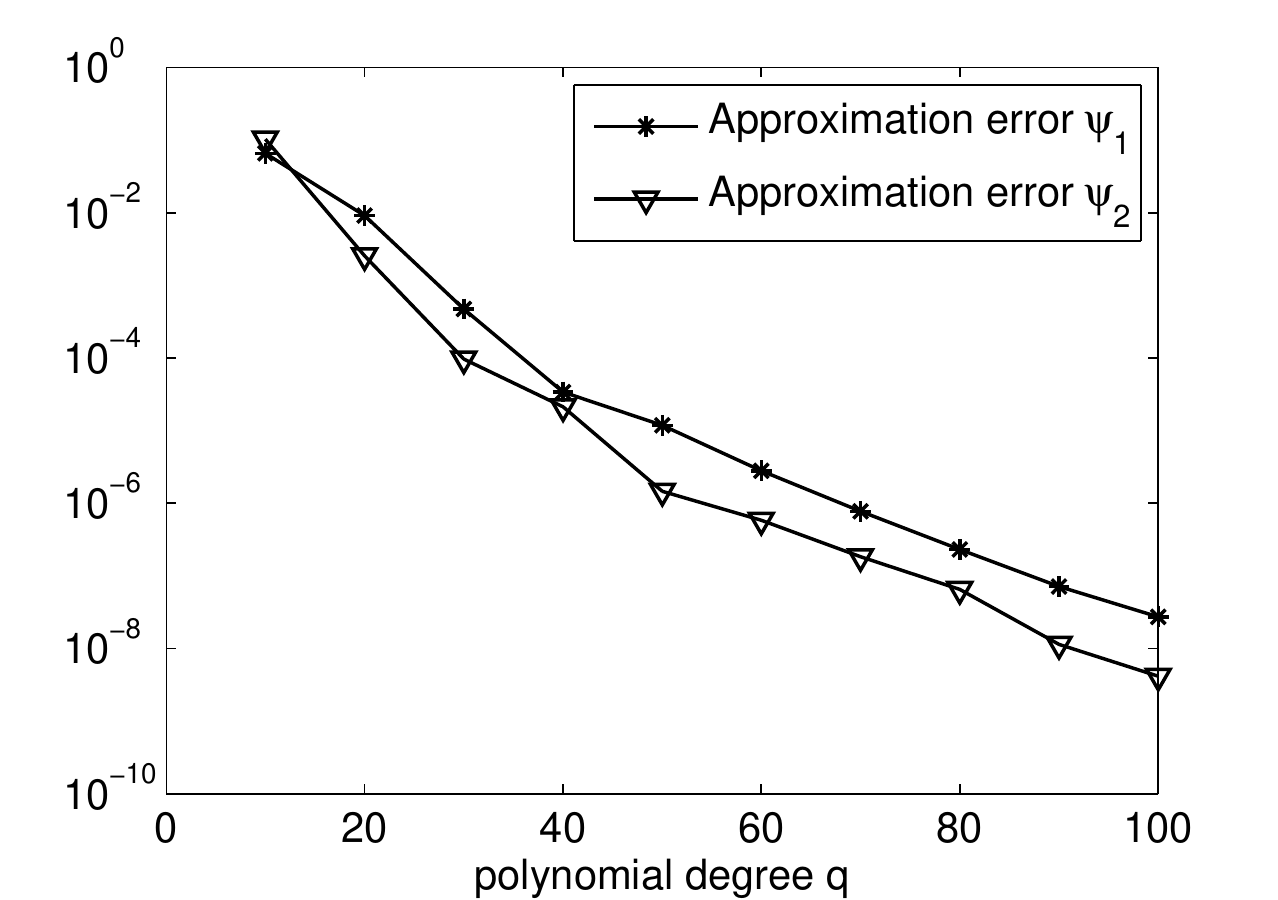}
\caption{$\|\psi_1-\psi_1^{\text{approx}}\|_{\infty}$ and $\|\psi_2-\psi_2^{\text{approx}}\|_{\infty}$ in dependence of $q$ for $m=1$. }
\label{Fig:psiErr}
\end{minipage}
\end{figure}

$\psi_{1}$ and $\psi_{2}$ are illustrated in Figure
\ref{Fig:psi}. Functions of type $\psi_{2}$ occur in the discretization
process if higher order methods in time are used. Although the higher order
basis functions $b_{3}$ and $b_{4}$ are more oscillatory than $b_{1}$ and
$b_{2}$, Figure \ref{Fig:psi} shows that the corresponding function $\psi_{2}$
is of similar shape than $\psi_{1}$ due to smoothing effect of the
integration.\newline Figure \ref{Fig:psiErr} shows the error that results from
the approximation of $\psi_{1}$ and $\psi_{2}$ with the Chebyshev
approximation \eqref{Chebapprox} of different polynomial degree $q$ on the
interval $[0,7]$, i.e. $l=1$. It becomes evident that the maximal pointwise
error decreases quickly with increasing $q$. However exponential convergence
cannot be observed due to the non-analyticity of $\psi_{1}$ and $\psi_{2}$. In
the following table we list the approximation errors for different values of
$m$ and $q$. They are chosen such that the original function has to be
evaluated 100 times to compute the approximation. Also from this table, we
conclude that the use of (moderately) high polynomial degrees in time does not
require a significantly higher number of quadrature points for the accurate
evaluation of the matrix entries \eqref{integrals2}.

\begin{table}[h]
\begin{center}%
\begin{tabular}
[c]{|c|c|c|c|}\hline
 $m$ & $q$ & $\Vert\psi_{1}-\psi_{1}^{\text{approx}}\Vert_{\infty}$ &
$\Vert\psi_{2}-\psi_{2}^{\text{approx}}\Vert_{\infty}$\\\hline
1 & 100 & $2.72\cdot10^{-8}$ & $4.16\cdot10^{-9}$\\
2 & 50 & $4.28\cdot10^{-8}$ & $1.88\cdot10^{-8}$\\
4 & 25 & $4.97\cdot10^{-8}$ & $2.60\cdot10^{-8}$\\
5 & 20 & $3.87\cdot10^{-8}$ & $1.34\cdot10^{-8}$\\
10 & 10 & $1.60\cdot10^{-7}$ & $2.19\cdot10^{-7}$\\
20 & 5 & $2.25\cdot10^{-5}$ & $1.14\cdot10^{-5}$\\
25 & 4 & $1.14\cdot10^{-4}$ & $4.99\cdot10^{-5}$\\
50 & 2 & $3.39\cdot10^{-3}$ & $1.43\cdot10^{-3}$\\\hline
\end{tabular}
\end{center}
\caption{Approximation errors for different values of $m$ and $q$.}
\end{table}

The table above shows that a low number of subintervals and a
modest polynomial degree is the best choice in terms of accuracy and
efficiency of the evaluation.


\subsection{Evaluation of the matrix entries}

\label{Sec:Entries} Let us assume that a triangulation $\mathcal{G}$ of
$\Gamma$ is given and that $\tau_{1},\tau_{2}\in\mathcal{G}$ are triangles of
size $O\left(  h\right)  $ in this triangulation. The computation of the
matrix entries $\eqref{integrals2}$ belonging to the matrix block
$\mathbf{A}_{k,i}$ requires the efficient approximation of integrals of the
form%
\begin{equation}
\int_{\tau_{1}}\int_{\tau_{2}}\frac{\varphi_{j}(y)\,\varphi_{l}(x)}{4\pi\Vert
x-y\Vert}\psi_{k,i}(\Vert x-y\Vert)\,d\Gamma_{y}d\Gamma_{x}.
\label{integrals3}%
\end{equation}
In order to evaluate \eqref{integrals3} we transform this integral to the
4-dimensional unit cube and apply tensor-Gauss-Legendre quadrature. In case
that $\tau_{1}$ and $\tau_{2}$ are identical, share a common edge or have a
common point we apply regularizing coordinate transformations (cf.
\cite{sauter2010boundary}) which remove the spatial singularity at $x=y$ via
the determinant of the Jacobian and also allow the use of standard
tensor-Gauss quadrature.\vspace{0pt}

The convergence analysis of tensor-Gauss-Legendre quadrature for integrals of
type \eqref{integrals3} is not straightforward since standard tools cannot be
used due to the non-analyticity of the involved integrands. Precise knowledge
about the growth behavior of the derivative of the integrands is necessary in
order to estimate the quadrature error. Since the derivatives of these
functions grow typically much faster than for analytic integrands, error
estimates must be used that use only lower order derivatives of the involved
functions (see \cite{Trefethen}). An analysis of the growth behavior of the
derivatives of the partition of unity function $\rho_{a,b,c}$ and the
corresponding quadrature error analysis was given in \cite{Sauter5}. The
analysis was extended to integrals of type \eqref{integrals3} in \cite{Schmid}
in the case that the triangles $\tau_{1}$ and $\tau_{2}$ have positive
distance.\newline Let $E_{n}$ denote the error of the tensor-Gauss-Legendre
quadrature approximation to the integral \eqref{integrals3}, where $n$
quadrature points in each direction are used (total number of quadrature
points: $n^{4}$).

\begin{theorem}
\label{THMQuadError} Let the triangles $\tau_{1}$ and $\tau_{2}$ in
\eqref{integrals3} have positive distance $D$ and let $\lambda\in(0,\frac
{2}{3})$. Then, there exists $n_{\lambda}\in\mathbb{N}$ such that for all
$n>n_{\lambda}$ it holds
\[
E_{n}\leq C\cdot\frac{\ln(n)^{\frac{1}{2}}}{\ln(n)-2}\cdot n^{-\lambda
\ln(n)+2}.
\]
The constants $C$ and $n_{\lambda}$ depend on the degrees of the involved
basis functions in space and time, on the distance $D$, and the size of the triangles.
\end{theorem}

\begin{proof}
The theorem follows directly from the results in \cite[Section 5.5]{Schmid}.
\end{proof}

Theorem \ref{THMQuadError} shows that the quadrature error decays
superalgebraically with respect to the number of quadrature nodes $n$. This
result cannot be improved to exponential convergence by a refined analysis (at
least when the error estimate in \cite{Trefethen} is used) and is in this
sense (asymptotically) sharp. In practical computations, however, it becomes
evident that the actual quadrature error decays considerably faster in a
preasymptotic range.\vspace{\baselineskip}

In the following we perform various numerical experiments which show the
performance of the quadrature scheme for integrals of type \eqref{integrals3}
(see \cite{Schmid, KhSaVe}). We distinguish between singular integrals
(identical panels, common edge) and regular integrals where the triangles
$\tau_{1}$ and $\tau_{2}$ have positive distance. Here we furthermore
distinguish between near field integrals where $\text{dist}(t,\tau)\sim h$ and
far field integrals where $\text{dist}(t,\tau)\sim O(1)$ (see
\cite{sauter2010boundary}). We use different triangles $\tau_{1}$ and
$\tau_{2}$ and different time grids to cover various situations. We consider
piecewise constant basis functions in space and denote by
\[
b_{t_{i}}(t):=\rho_{t_{i},t_{i+1},t_{i+2}}(t)P_{1}\left(  2\frac{t-t_{i}%
}{t_{i+2}-t_{i}}-1\right)
\]
the basis functions in time that will be used in the experiments. The
resulting integrals which will be approximated by tensor-Gauss-Legendre
quadrature (after a (regularizing) transformation to the reference element)
are therefore of the form
\begin{equation}
\int_{\tau_{1}}\int_{\tau_{2}}\frac{\psi_{t_{i}}^{t_{j}}(\Vert x-y\Vert)}%
{4\pi\Vert x-y\Vert}\,d\Gamma_{y}d\Gamma_{x}\quad\text{with}\quad\psi_{t_{i}%
}^{t_{j}}(r):=\int_{0}^{T}\dot{b}_{t_{i}}(t-r)b_{t_{j}}(t)dt.
\label{IntegralsToApprox}%
\end{equation}
More precisely we consider the following settings:\\[3mm]\textbf{Case 1:}
\emph{Identical panels, completely enlighted}\\[-2mm]

\hspace{5mm}Triangles:
\[
\tau_{1} = \tau_{2} = \text{conv}\left\lbrace (0,0,0)^{\operatorname{T}%
},(1,0,0)^{\operatorname{T}},(1,1,0)^{\operatorname{T}} \right\rbrace .
\]
\hspace{5mm} Time grid:
\[
t_{0} = 0,\quad t_{1} = 1.2,\quad t_{2} = 2,\quad,t_{3} = 2.9
\]
\hspace{5mm} and the integrand $\psi_{t_{0}}^{t_{1}}$ in
\eqref{IntegralsToApprox} such that $\text{supp}\,\psi_{t_{0}}^{t_{1}} = [0,
2.9] $.\\[3mm]\textbf{Case 2:} \emph{Panels with common edge, partially
enlighted}\\[-2mm]

\hspace{5mm}Triangles:
\begin{align*}
\tau_{1}  &  = \text{conv}\left\lbrace (0,0,0)^{\operatorname{T}%
},(1,0,0)^{\operatorname{T}},(1,1,0)^{\operatorname{T}} \right\rbrace ,\\
\tau_{2}  &  = \text{conv}\left\lbrace (0,0,0)^{\operatorname{T}%
},(1,0,0)^{\operatorname{T}},(1,-1,0.5)^{\operatorname{T}} \right\rbrace .
\end{align*}
\hspace{5mm}Time grid:
\[
t_{0} = 0,\quad t_{1} = 1.1,\quad t_{2} = 2.1,\quad t_{3} = 2.9,\quad t_{4} =
4,\quad t_{5} = 5
\]
\hspace{5mm}and the integrand $\psi_{t_{0}}^{t_{3}}$ in
\eqref{IntegralsToApprox} such that $\text{supp}\,\psi_{t_{0}}^{t_{3}} = [0.8,
5]$.\\[3mm]\textbf{Case 3:} \emph{Panels with positive distance, near field,
completely enlighted}\\[-2mm]

\hspace{5mm}Triangles:
\begin{align*}
\tau_{1}  &  = \text{conv}\left\lbrace (0,0,0)^{\operatorname{T}%
},(1,0,0)^{\operatorname{T}},(1,1,0)^{\operatorname{T}} \right\rbrace ,\\
\tau_{2}  &  = \text{conv}\left\lbrace (1,0,0)^{\operatorname{T}%
},(1,0.9,0)^{\operatorname{T}},(0,1,0.2)^{\operatorname{T}} \right\rbrace +
(2,2,2)^{\operatorname{T}}.
\end{align*}
\hspace{5mm}Time grid:
\[
t_{0} = 0,\quad t_{1} = 1.2,\quad t_{2} = 2.1,\quad t_{3} = 3.9,\quad t_{4} =
5.1,\quad t_{5} = 6
\]
\hspace{5mm}and the integrand $\psi_{t_{0}}^{t_{3}}$ in
\eqref{IntegralsToApprox} such that $\text{supp}\,\psi_{t_{0}}^{t_{3}} = [1.8,
6]$.\\[3mm]\textbf{Case 4:} \emph{Panels with positive distance, far field,
partially enlighted}\\[-2mm]

\hspace{5mm}Triangles:
\begin{align*}
\tau_{1}  &  = \text{conv}\left\lbrace (0,0,0)^{\operatorname{T}%
},(1,0,0)^{\operatorname{T}},(1,1,0)^{\operatorname{T}} \right\rbrace ,\\
\tau_{2}  &  = \text{conv}\left\lbrace (1,0,0)^{\operatorname{T}%
},(1,0.9,0)^{\operatorname{T}},(0,1,0.2)^{\operatorname{T}} \right\rbrace +
(20,20,20)^{\operatorname{T}}.
\end{align*}
\hspace{5mm}Time grid:
\[
t_{0} = 0,\quad t_{1} = 1.2,\quad t_{2} = 2.1,\quad t_{3} = 30.5,\quad t_{4} =
31.6,\quad t_{5} = 32.6
\]
\hspace{5mm}and the integrand $\psi_{t_{0}}^{t_{3}}$ in
\eqref{IntegralsToApprox} such that $\text{supp}\,\psi_{t_{0}}^{t_{3}} =
[28.4, 32.6]$.\\[3mm]

Note that the time stepsizes were chosen such that they correspond
approximately to the diameter of the triangles.\newline Figure \ref{quadError}
shows the convergence of tensor-Gauss-Legendre quadrature for integrals of
type \eqref{IntegralsToApprox} for the different cases described above. It
becomes evident that the error decays quickly in all four cases, especially in
the preasymptotic regime. As Theorem \ref{THMQuadError} predicts, exponential
convergence cannot be observed for medium and higher numbers of quadrature
nodes for such smooth but non-analytic integrands. In
Section \ref{NumEx} we report on numerical experiments for studying the
influence of the quadrature error to the overall accuracy. It turns out that
the necessary number of quadrature nodes is very moderate.

\begin{figure}[th]
\centering
\includegraphics[width=1.0\textwidth]{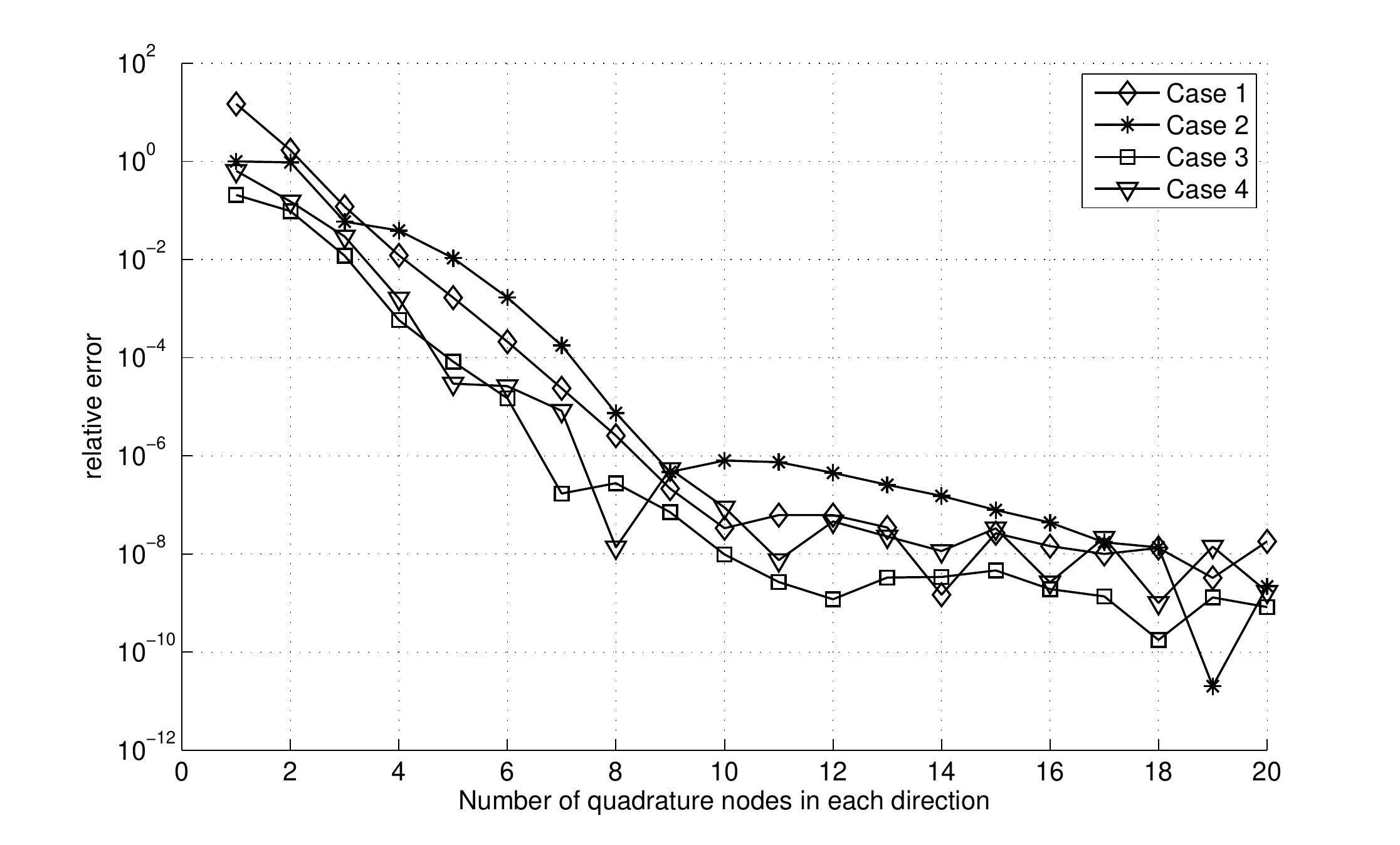}\caption{Convergence of
tensor Gauss-Legendre quadrature for integrals of type
\eqref{IntegralsToApprox} for the cases 1-4.}%
\label{quadError}%
\end{figure}


\subsection{Computation of the $H^{-1/2,-1/2}(\mathbb{S}^{2}\times
[0,T])$-norm}

\label{Sec:Norm} In Section \ref{NumEx} we perform numerical experiments for a
spherical scatterer, i.e. $\Gamma=\mathbb{S}^{2}$, and special right-hand
sides of the form $g(x,t)=g(t)Y_{n}^{m},t\in\lbrack0,T]$, where $Y_{n}^{m}$
are the spherical harmonics of degree $n$ and order $m$. In this case the
exact solution of the scattering problem also decouples in space and time and
is of the form
\[
\phi_{\text{exact}}(x,t)=\phi_{\text{exact}}(t)Y_{n}^{m}\quad\text{with}%
\quad(x,t)\in\mathbb{S}^{2}\times\lbrack0,T].
\]
Explicit representations of these exact solutions were derived in
\cite{Sauter4} and will be used as reference solutions to test the numerical
algorithm. In order to estimate the error of the Galerkin approximation
$\phi_{\text{Galerkin}}$ a computation of the $H^{-1/2,-1/2}(\mathbb{S}%
^{2}\times\lbrack0,T])$-norm is necessary. Since this norm is difficult to
compute directly we use the sesquilinear form \eqref{VarForm} with its
coercivity property \eqref{coercivity} in order to obtain an upper bound for
this norm (up to a constant).\\
In this article we consider only boundary
element meshes consisting of flat triangles whose union defines a polyhedral
surface approximation $\Gamma_{h}$ of the original surface $\Gamma$. Hence,
the exact Galerkin solution is perturbed due to this surface approximation and
we denote the sesquilinear form on $\Gamma_{h}$ by $a_{h}\left(  \cdot
,\cdot\right)  $. In order to compare the exact solution with the Galerkin
solution we will project the exact solution $\phi_{\operatorname*{exact}}$ to
the approximate surface $\Gamma_{h}$ resulting in a function $\phi
_{\text{exact}}^{h}$ on $\Gamma_{h}$. To measure the difference $\phi
_{\operatorname*{exact}}^{h}-\phi_{\operatorname*{Galerkin}}$ in an
approximated (squared) energy norm we plug it into the sesquilinear form
$a_{h}(\cdot,\cdot)$. Let us assume as before that the Galerkin solution is
defined on $\Gamma_{h}$ and that
\[
\phi_{\text{Galerkin}}(x,t)\in V_{\text{Galerkin}}:=\text{span}\left\{
b_{i}(t)\varphi_{j}(x),1\leq i\leq L,1\leq j\leq M\right\}  .
\]
Since in Section \ref{NumEx} we will mainly focus on the properties of the
temporal discretization we introduce a discrete space on a fine time grid
(using possibly higher order basis functions in time) which uses the same
basis functions in space as $V_{\text{Galerkin}}$:
\[
V_{\text{Galerkin}}^{\text{fine}}:=\text{span}\left\{  \tilde{b}_{i}%
(t)\varphi_{j}(x),1\leq i\leq\tilde{L},1\leq j\leq M\right\}  \subset
H^{-1/2,-1/2}(\Gamma_{h}\times\lbrack0,T]).
\]
We now approximate $\phi_{\text{exact}}$ and $\phi_{\text{Galerkin}}$ with
functions $\phi_{\text{exact}}^{h,\tilde{L}},\phi_{\operatorname{Galerkin}%
}^{\tilde{L}}\in V_{\text{Galerkin}}^{\text{fine}}$ in order to efficiently
evaluate the associated sesquilinear form.\vspace{0pt}

For the spatial approximation of $\phi_{\text{exact}}$ we note that in the
case of piecewise constant basis functions in space every $\varphi_{j},1\leq
j\leq M$ is associated with a triangle $\Delta_{j}=\text{conv}\left\{
A_{j},B_{j},C_{j}\right\}  $, where $A_{j},B_{j},C_{j}\in\Gamma$. An
approximation of $\phi_{\text{exact}}$ defined on $\Gamma_{h}\times
\lbrack0,T]$ is then defined by
\[
\phi_{\text{exact}}^{h}(x,t):=\phi_{\text{exact}}(t)\sum_{j=1}^{M}c_{j}%
^{h}\varphi_{j}(x),\quad\text{with}\quad c_{j}^{h}=Y_{n}^{m}|_{D_{j}}%
\quad\text{where}\quad D_{j}=\frac{A_{j}+B_{j}+C_{j}}{\Vert A_{j}+B_{j}%
+C_{j}\Vert}.
\]
In the case of piecewise linear basis functions in space every $\varphi
_{j},1\leq j\leq M$, is associated with a node $C_{j}\in\Gamma$ of the spatial
mesh. A suitable approximation of $\phi_{\text{exact}}$ defined on $\Gamma
_{h}\times\lbrack0,T]$ is in this case defined by
\[
\phi_{\text{exact}}^{h}(x,t):=\phi_{\text{exact}}(t)\sum_{j=1}^{M}c_{j}%
^{h}\varphi_{j}(x),\quad\text{with}\quad c_{j}=Y_{n}^{m}|_{C_{j}}.
\]
In order to obtain an approximation of $\phi_{\text{exact}}^{h}(x,t)$ in the
space $V_{\text{Galerkin}}^{\text{fine}}$ we further approximate the temporal
part $\phi_{\text{exact}}(t)$ with its best $L_{2}$-approximation in
$\text{span}\left\{  \tilde{b}_{i},1\leq i\leq\tilde{L}\right\}  $ on the fine
time grid. This leads to
\[
\phi_{\text{exact}}(x,t)\approx\phi_{\text{exact}}^{h,\tilde{L}}%
(x,t):=\sum_{i=1}^{\tilde{L}}\sum_{j=1}^{M}c_{i}^{\tilde{L}}c_{j}^{h}%
\varphi_{j}(x)\tilde{b}_{i}(t).
\]

Finally, the function
\[
\phi_{\operatorname{Galerkin}}(x,t)=\sum_{i=1}^{L}\sum_{j=1}^{M}\alpha_{i}%
^{j}\varphi_{j}(x)b_{i}(t)=\sum_{j=1}^{M}\varphi_{j}(x)\underbrace{\sum
_{i=1}^{L}\alpha_{i}^{j}b_{i}(t)}_{\phi_{j}(t)}%
\]
has to be approximated with a function in $V_{\text{Galerkin}}^{\text{fine}}$.
For this we approximate the function $\phi_{j}(t)$ for every $1\leq j\leq M$
again with its best $L_2$-approximation in
$\text{span}\left\{  \tilde{b}_{i},1\leq i\leq\tilde{L}\right\}  $. This
defines coefficients $\tilde{\alpha}_{i}^{j}$ such that
\[
\phi_{\operatorname{Galerkin}}(x,t)\approx\sum_{i=1}^{\tilde{L}}\sum_{j=1}%
^{M}\tilde{\alpha}_{i}^{j}\varphi_{j}(x)\tilde{b}_{i}(t)=:\phi
_{\operatorname{Galerkin}}^{\tilde{L}}(x,t).
\]

In order to estimate the error of the Galerkin approximation we denote
$\text{err}_{G}:=\Vert\phi_{\text{Galerkin}}^{\tilde{L}}-\phi_{\text{exact}%
}^{h,\tilde{L}}\Vert_{H^{-1/2,-1/2}(\Gamma_{h}\times\lbrack0,T])}$. Since
\[
\phi_{\text{Galerkin}}^{\tilde{L}}-\phi_{\text{exact}}^{h,\tilde{L}}%
=\sum_{i=1}^{\tilde{L}}\sum_{j=1}^{M}(\tilde{\alpha}_{i}^{j}-c_{i}^{\tilde{L}%
}c_{j}^{h})\varphi_{j}(x)\tilde{b}_{i}(t),
\]
the coercivity estimate \eqref{coercivity} leads to
\begin{align*}
\text{err}_{G}^{2} &  \lesssim a(\phi_{\text{Galerkin}}^{\tilde{L}}%
-\phi_{\text{exact}}^{h,\tilde{L}},\phi_{\text{Galerkin}}^{\tilde{L}}%
-\phi_{\text{exact}}^{h,\tilde{L}})\\
&  =\sum_{i=1}^{\tilde{L}}\sum_{k=1}^{\tilde{L}}\sum_{j=1}^{M}\sum_{l=1}%
^{M}\int_{0}^{T}\int_{\Gamma}\int_{\Gamma}(\tilde{\alpha}_{i}^{j}%
-c_{i}^{\tilde{L}}c_{j}^{h})(\tilde{\alpha}_{k}^{l}-c_{k}^{\tilde{L}}c_{l}%
^{h})\frac{\dot{\tilde{b}}_{i}(t-\Vert x-y\Vert)\varphi_{j}(y)\tilde{b}%
_{k}(t)\varphi_{l}(x)}{4\pi\Vert x-y\Vert}d\Gamma_{y}d\Gamma_{x}dt\\
&  =\sum_{i=1}^{\tilde{L}}\sum_{k=1}^{\tilde{L}}\sum_{j=1}^{M}\sum_{l=1}%
^{M}(\tilde{\alpha}_{i}^{j}-c_{i}^{\tilde{L}}c_{j}^{h})\mathbf{\tilde{A}%
}_{k,i}(j,l)(\tilde{\alpha}_{k}^{l}-c_{k}^{\tilde{L}}c_{l}^{h})\\
&  =(\boldsymbol{\tilde{\alpha}}-\mathbf{c})^{\operatorname{T}}\underline
{\underline{\mathbf{\tilde{A}}}}(\boldsymbol{\tilde{\alpha}}-\mathbf{c})
\end{align*}
with
\[
\boldsymbol{\tilde{\alpha}}=\left(  \boldsymbol{\tilde{\alpha}}_{i}\right)
_{i=1}^{\tilde{L}},\quad\text{where}\quad\boldsymbol{\tilde{\alpha}}%
_{i}(j)=\left(  \tilde{\alpha}_{i}^{j}\right)  _{j=1}^{\tilde{M}}%
\]
and in the same way
\[
\mathbf{c}=\left(  \boldsymbol{c}_{i}\right)  _{i=1}^{\tilde{L}}%
,\quad\text{where}\quad\mathbf{c}_{i}(j)=\left(  c_{i}^{\Delta t}c_{j}%
^{h}\right)  _{j=1}^{M}.
\]
We therefore use the quantities
\begin{equation}
\text{err}(\phi_{\text{exact}},\phi_{\text{Galerkin}}):=\sqrt
{(\boldsymbol{\tilde{\alpha}}-\mathbf{c})^{\operatorname{T}}\underline
{\underline{\mathbf{\tilde{A}}}}(\boldsymbol{\tilde{\alpha}}-\mathbf{c}%
)}\label{errorMeasure}%
\end{equation}
and
\begin{equation}
\text{err}_{\text{rel}}(\phi_{\text{exact}},\phi_{\text{Galerkin}}):=\sqrt
{\frac{(\boldsymbol{\tilde{\alpha}}-\mathbf{c})^{\operatorname{T}}\underline
{\underline{\mathbf{\tilde{A}}}}(\boldsymbol{\tilde{\alpha}}-\mathbf{c}%
)}{  \mathbf{c}^{\operatorname{T}}\underline
{\underline{\mathbf{\tilde{A}}}}\mathbf{c} } }\label{relerrorMeasure}%
\end{equation}
as measures for the error of our Galerkin approximation.

\begin{remark}
Since the space on which the sesquilinear form $a(\cdot,\cdot)$ is coercive
differs from the space where it is continuous (cf. \cite{HaDuong}) it is an
open question if the error measure \eqref{errorMeasure} is actually equivalent
to the $H^{-1/2,-1/2}(\Gamma_{h}\times\lbrack0,T])$-norm or if it only
represents an upper bound (up to a constant).
\end{remark}

\section{A Posteriori Error Estimation in Time}

\label{Sec:ErrEst}

In this section we want to introduce a suitable a posteriori error
estimator in time. Since in practice the solution of the
boundary integral equation \eqref{BIE} might be rough (oscillatory or
non-smooth) at certain times and rather smooth at other times it is in general
not optimal to choose a fine time grid with constant step size everywhere on
the time interval $[0,T]$ in order to resolve such a solution. Instead, a
suitably chosen time grid that is adapted to the local irregularities of the
solution with a lower number of variable time steps might be advantageous in
this case and can lead to a more efficient scheme.\newline Since it is in
general not known in advance where the solution is rough the numerical method
should detect automatically where a local refinement of the time grid is
necessary. This is done via the above mentioned a posteriori error estimator
which computes local quantities $\left(  \eta_{i}\right)  _{i=1}^{L}$ that are
associated with the local error of the Galerkin approximation. These
quantities serve as refinement indicators in the adaptive scheme.\newline%
Note that the Galerkin discretization in time is not a
time stepping method but has to be solved for the entire time mesh as a
coupled system. The (localized) error estimator then indicates which time
intervals should be marked for refinement (cf. Figure \ref{Fig:Adap}). Our
numerical experiments indicate that, for problems in wave propagation, it is
essential that an a posteriori error indicator examines all time steps in
history instead of trying to determine within a \textit{time stepping} method
which interval in the history has to be refined and to set back the current
time step to the relevant one in the history.\vspace{3pt}

\begin{figure}[h]
\begin{center}
\includegraphics[width=0.9\textwidth]{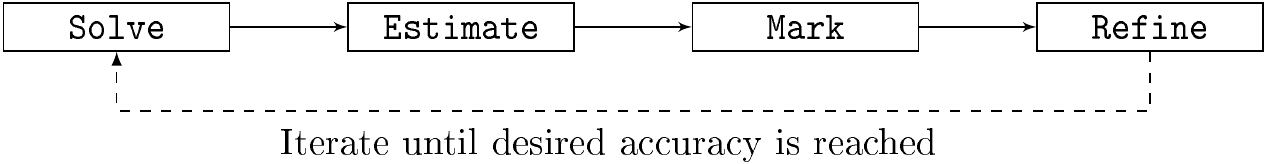}
\end{center}
\caption{Adaptive strategy}%
\label{Fig:Adap}%
\end{figure}

The proposed algorithm currently uses the same time grid everywhere on the
spatial domain in order to compute an approximation. Since the optimal time
grid at different points of the scatterer might not coincide, we compute
suitable refinements of the time grid at different points $x_{0}\in\Gamma$ and
solve the scattering problem in the next step on the union of the proposed
time grids. More precisely we perform the following steps in the time-adaptive algorithm:

\begin{description}

\item[Solve:] Solve the full problem \eqref{BIE} approximately for a
given triangulation and time grid.

\item[Estimate:] Choose a finite set of points $\Xi\subset\Gamma$ and
compute, for each $x_{0}\in\Xi$, refinement indicators $\eta_{x_{0},i}, 1\leq
i\leq L$, which are connected to the time step $t_{i}$ (also denoted as
$t_{x_{0},i})$.

\item[Mark:] Choose a threshold $\alpha\in(0,1)$ and mark for each
$x_{0}\in\Xi$ all time steps $t_{x_{0},k}$ such that $\eta_{x_{0},k}\geq
\alpha\max_{1\leq i\leq L}(\eta_{x_{0},i})$.

\item[Refine:] For fixed $x_{0}\in\Xi$ insert additional timesteps in
the middle of the subintervals $[t_{x_{0},k-1},t_{x_{0},k}]$ and $[t_{x_{0}%
,k},t_{x_{0},k+1}]$, where $t_{x_{0},k}$ is a marked element. This leads to a
refined time grid $\Delta_{x_{0}}$ for $x_{0}\in\Xi$. Choose $\Delta=
\bigcup_{x_{0}\in\Xi}\Delta_{x_{0}}$ as the new time grid and iterate the
procedure until a desired accuracy is achieved.
\end{description}

It remains to define suitable refinement indicators $\eta_{x_{0},i}$. Note
that for the retarded single layer potential we have (see \cite[Thm. 3]{HaDuong})
\[
S:H^{1,-1/2,-1/2}(\Gamma\times [0,T])\rightarrow H^{1/2,1/2}(\Gamma
\times [0,T]),
\]
where
\[
H^{1,-1/2,-1/2}(\Gamma\times [0,T]):=\left\lbrace \phi ; \dot{\phi}\in H^{-1/2,-1/2}(\Gamma\times [0,T])\right\rbrace
\]
and
\[
H^{1/2,1/2}\left(  \Gamma\times\left[  0,T\right]  \right)  :=L^{2}\left(
\left[  0,T\right]  ,H^{1/2}\left(  \Gamma\right)  \right)  \cap
H^{1/2}\left(  \left[  0,T\right]  ,L^{2}\left(  \Gamma\right)  \right)  .
\]
\begin{remark}Recall that in practical computations we solve the variational equation \eqref{VarForm} approximately for $\dot{\phi}$ and obtain an approximate solution $\phi$ of the boundary integral equation in a postprocessing step. The refinement indicators that we will introduce are therefore based on the residual $S\dot{\phi}-\dot{g}$ which is in $H^{1/2,1/2}(\Gamma\times [0,T])$ due to the mapping properties of $S$.
More precisely we have chosen the efficient and reliable a
posteriori error estimator for operators of negative order that was originally
developed for elliptic problems (see \cite{Faermann}) and adapted this
estimator to the retarded potential integral equations.
\end{remark}
The error
estimators are based on an explicit representation of the $H^{1/2}$-seminorm.
For an interval $\omega\subset\mathbb{R}$ it holds
\[
|\xi|_{H^{1/2}(\omega)}^{2}=\int_{\omega}\int_{\omega}\frac{|\xi(t)-\xi
(\tau)|^{2}}{|t-\tau|^{2}}d\tau dt.
\]
For an arbitrary point $x_{0}\in\Xi$ on the boundary $\Gamma$ we define the
residual
\[
r_{x_{0}}(t):=S\dot{\phi}_{\text{Galerkin}}(x_{0},t)-\dot{g}(x_{0},t)
\]
of the Galerkin approximation. Let a time grid as in \eqref{timegrid} be given
and define
\[
\omega_{1}=[t_{0},t_{1}],\qquad\omega_{i}=[t_{i-2},t_{i}],\,2\leq i\leq
l-1,\qquad\omega_{l}=[t_{l-2},t_{l-1}].
\]
Then, local (temporal) refinement indicators are given by
\begin{equation}
\eta_{x_{0},i}:=|r_{x_{0}}|_{H^{1/2}(\omega_{i})}=\int_{\omega_{i}}%
\int_{\omega_{i}}\frac{|r_{x_{0}}(t)-r_{x_{0}}(\tau)|^{2}}{|t-\tau|^{2}}d\tau
dt,\quad i=1,\ldots,l.\label{refInd}%
\end{equation}
Due to the Lipschitz-continuity of the residual $r_{x_{0}}$, the integrand in
\eqref{refInd} is non-singular. However, due to the removable singularity the
double integral has to be evaluated with care. Here, we apply simple
coordinate transformations which move the singularity to the boundary of the
unit square and evaluate \eqref{refInd} using tensor-Gauss-Legendre quadrature
rules. Let
\[
\tilde{r}_{x_{0}}(t,\tau):=\frac{|r_{x_{0}}(t)-r_{x_{0}}(\tau+t)|^{2}}%
{|\tau|^{2}},\qquad\omega_{i}=[c,d]
\]
and
\[
\chi_{1}:t\mapsto(d-c)t+c,\qquad\chi_{2}:t\mapsto(d-c)t.
\]
Then
\begin{align*}
|r_{x_{0}}|_{H^{1/2}(\omega_{i})} &  =\int_{\omega_{i}}\int_{\omega_{i}}%
\frac{|r_{x_{0}}(t)-r_{x_{0}}(\tau)|^{2}}{|t-\tau|^{2}}d\tau dt\\
&  =\int_{c}^{d}\int_{0}^{d-t}\tilde{r}_{x_{0}}(t,\tau)d\tau dt+\int_{c}%
^{d}\int_{c-t}^{0}\tilde{r}_{x_{0}}(t,\tau)d\tau dt\displaybreak[0]\\
&  =\int_{c}^{d}\int_{0}^{t-c}\tilde{r}_{x_{0}}(-t+c+d,\tau)d\tau dt+\int
_{c}^{d}\int_{0}^{t-c}\tilde{r}_{x_{0}}(t,-\tau)d\tau dt\\
&  =(d-c)^{2}\int_{0}^{1}\int_{0}^{t}\underbrace{\tilde{r}_{x_{0}}\left(
-\chi_{1}(t)+c+d,\chi_{2}(\tau)\right)  +\tilde{r}_{x_{0}}(\chi_{1}%
(t),-\chi_{2}(\tau))}_{=:\,\tilde{\tilde{r}}_{x_{0}}(t,\tau)}\,\,d\tau dt.
\end{align*}
With the Duffy-transform $(t,\tau)\mapsto(t,t\tau)$ we map the triangle to the
unit square and obtain
\begin{equation}
|r_{x_{0}}|_{H^{1/2}(\omega_{i})}=(d-c)^{2}\int_{0}^{1}\int_{0}^{1}%
\tilde{\tilde{r}}_{x_{0}}(t,t\tau)t\,\,d\tau dt.\label{refIndtrans}%
\end{equation}
The double integral in \eqref{refIndtrans} can be approximated efficiently
using tensor Gauss-Legendre quadrature since the integrand is well defined in
the interior of the unit square.

\section{Numerical Experiments}

\label{NumEx}

\emph{Convergence tests}\vspace{\baselineskip}

In this section we present the results of numerical experiments. In order to
test the convergence of the method we solve the boundary integral equation
\eqref{BIE} for a spherical scatterer, i.e., $\Gamma= \mathbb{S}^{2}$ in the
time interval $[0,1]$. In a first experiment we consider the purely
time-dependent right-hand side
\begin{equation}
g(x,t) = t^{6}\operatorname{e}^{-4t},\quad(x,t)\in\mathbb{S}^{2}\times[0,1].
\label{rhs1}%
\end{equation}
In this simple scenario the exact solution of the scattering problem is known
explicitly (cf. \cite{Sauter4, Sauter5, Ban1}) and is given by
\begin{equation}
\label{exactSol}\phi(x,t) = 2\partial_{t} g(x,t).
\end{equation}
In a second experiment we consider the right-hand side
\begin{equation}
g(x,t) =g(t)Y_{1}^{0}:= t\sin(3t)^{2}\operatorname{e}^{-t}Y_{1}^{0}%
,\quad(x,t)\in\mathbb{S}^{2}\times[0,1], \label{rhs2}%
\end{equation}
where $Y_{1}^{0}$ is a spherical harmonic of degree $1$ and order $0$. The
exact solution of the problem is in this case given by
\[
\phi(x,t) = \left[  2\partial_{t} g(t) + 2\int_{0}^{t} \sinh(\tau)\partial_{t}
g(t-\tau)d\tau\right]  Y_{1}^{0}.
\]
For both configurations we discretize the scatterer using 2568 triangles and
approximate the solution in space with piecewise linear basis functions,
resulting in 1286 degrees of freedom in space.\newline The convergence of the
method with respect to the stepsize $\Delta t$ is depicted in Figure
\ref{ErrPlot} for different orders of the time discretization. The error was
computed using the error measure from Section \ref{Sec:Norm}.

\begin{figure}[th]
\centering
\subfigure[$g(x,t) = t^6\operatorname{e}^{-4t}$]{
\centering
\includegraphics[width=0.47\textwidth]{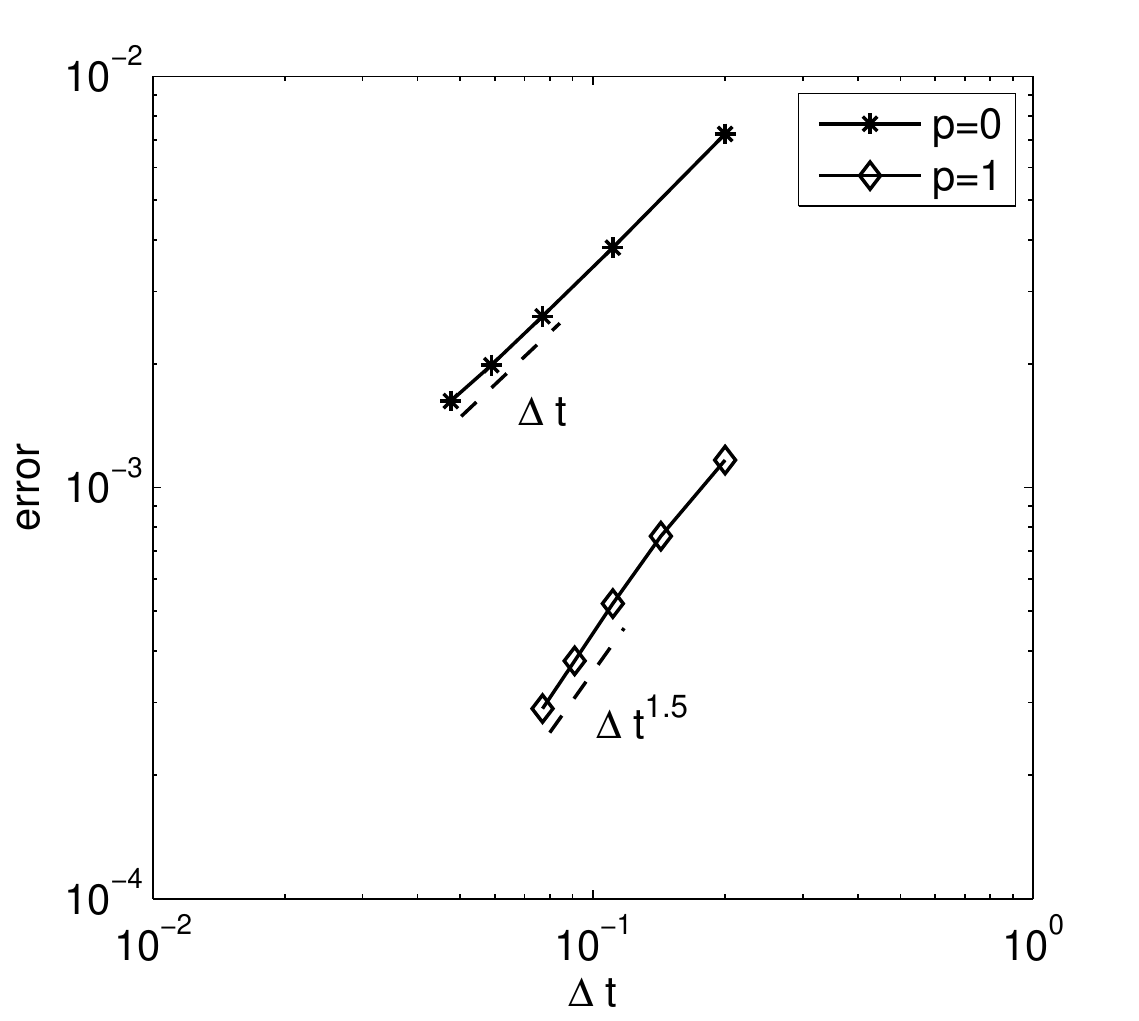}
} \hfil
\subfigure[$g(x,t) = t\sin(3t)^2\operatorname{e}^{-t}Y_1^0$]{
\centering
\includegraphics[width=0.47\textwidth]{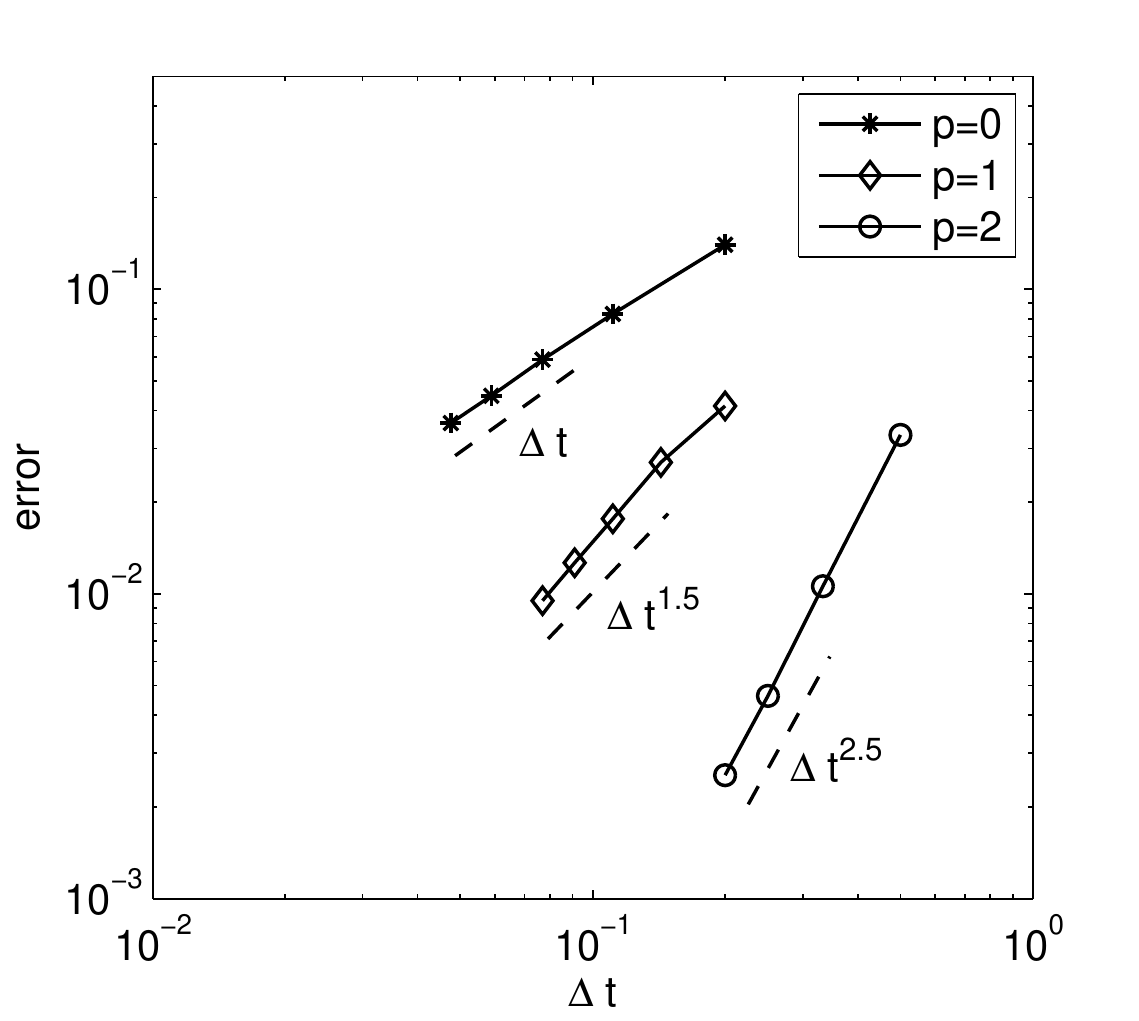}
}\caption{Convergence plots with respect to the stepsize in time. }%
\label{ErrPlot}%
\end{figure}

Theoretical convergence rates of the Galerkin approach
using piecewise polynomial basis functions were investigated in \cite{Bam1}
for $p>0$. Since our PUM basis functions have the same approximation
properties as the classical basis functions, the numerical experiments raise
the important question whether these theoretical error bounds are sharp in
general and the considered case of scattering from a sphere has special
properties or, possibly, the discrete evaluation of the energy norm gains
from, e.g., superconvergence properties
. This will be a topic of future
investigations.\newline

\emph{Influence of the quadrature order}\vspace{\baselineskip}

In Section \ref{Sec:Entries} we showed that the entries of the boundary
element matrix can be computed accurately with tensor Gauss quadrature rules.
Here we want to test the influence of the quadrature order on the overall
accuracy of the approximation. As an example we choose again a spherical
scatterer that is discretized using 616 triangles. We consider the time
interval $[0,5]$ which is subdivided into 20 equidistant subintervals. Note
that the configuration was chosen such that the stepsize in time corresponds
to the average diameter of the triangles. For the approximation we use
piecewise linear basis functions in space and time (i.e. $p=1$). As right-hand
side we choose a single Gaussian bump that travels in $x_{1}$ direction:
\begin{equation}
\label{Gaussbump}g(x,t) = \cos(t-x_{1})\operatorname{e}^{-6(t-x_{1}-5)^{2}},
\quad x = (x_{1},x_{2},x_{3})^{\operatorname{T}}%
\end{equation}

In the following we compute the arising boundary element matrices with
different accuracies. With $n_{\text{sing}},n_{\text{near}},n_{\text{far}}$ we
denote the number of quadrature points that are used in each direction for the
singular (regularized), the regular near field and the regular far field
integrands, respectively (cf. \cite{sauter2010boundary}). As a reference
solution we compute an approximation $\phi_{\text{high}}$ with $n_{\text{sing}%
}=20$, $n_{\text{near}}=15$ and $n_{\text{far}}=12$ on the same temporal and spatial grid mentioned above such that the discretization error is not visible. In Table \ref{quadAcc}
the results for different numbers of quadrature nodes are depicted. We measure
the error between $\phi_{\text{high}}$ and the Galerkin solution using lower
number of quadrature nodes in the error measure of Section \ref{Sec:Norm} and
in the $L^{2}([0,5],L^{2}(\Gamma))$-norm. \vspace{0pt}

\begin{table}[th]
\begin{center}%
\begin{tabular}
[c]{c|c|c|c|c}%
$n_{\text{sing}}$ & $n_{\text{near}}$ & $n_{\text{far}}$ & $\text{err}%
_{\text{rel}}(\phi_{\text{high}},\phi_{\text{Galerkin}})$ & rel. $L^{2}%
$-error\\\hline
10 & 8 & 6 & $1.86\cdot10^{-6}$ & $1.86\cdot10^{-6}$\\
8 & 6 & 5 & $1.43\cdot10^{-5}$ & $1.26\cdot10^{-5}$\\
6 & 5 & 4 & $1.03\cdot10^{-4}$ & $8.74\cdot10^{-5}$\\
5 & 4 & 3 & $5.36\cdot10^{-4}$ & $4.58\cdot10^{-4}$\\
5 & 3 & 3 & $1.43\cdot10^{-3}$ & $1.40\cdot10^{-3}$\\
4 & 3 & 3 & $1.87\cdot10^{-3}$ & $1.81\cdot10^{-3}$\\
4 & 3 & 2 & $2.48\cdot10^{-3}$ & $2.74\cdot10^{-3}$
\end{tabular}
\end{center}
\caption{Influence of quadrature on the accuracy of the Galerkin approximation
}%
\label{quadAcc}%
\end{table}

It becomes evident that a low number of quadrature nodes is sufficient to
compute stable and reasonably accurate solutions. Note that the results
obtained in Table \ref{quadAcc} depend on the CFL number. Whereas a large CFL
number is unproblematic with regard to the quadrature problem, a small CFL
number, i.e. the step size in time is much smaller than the diameter of the
triangles, typically requires a higher number of spatial quadrature nodes in
order to obtain accurate solutions. \\

\emph{Long term stability}\vspace{\baselineskip}

In order to test the stability of the method for a longer time interval we
consider again a spherical scatterer and solve problem \eqref{BIE} for the
right-hand side $g(x,t)=t^{4}\operatorname{e}^{-2t}$ for $T=40$. We discretize
the time interval using 120 equidistant timesteps and local polynomial
approximation spaces of degree $p=1$ resulting in 239 degrees of freedom in
time. The sphere is discretized with 616 triangles, which leads to 310 degrees
of freedom if piecewise linear approximation in space is used. The Galerkin
solution at $x=(1,0,0)^{\text{T}}$ is depicted in Figure \ref{long_time}. We
compare this result with the exact solution of the problem and with a
numerical solution that is obtained using BDF2-convolution quadrature using
also 120 time steps for the time discretization.

\begin{figure}[ht]
\centering
\includegraphics[width=1.0\textwidth]{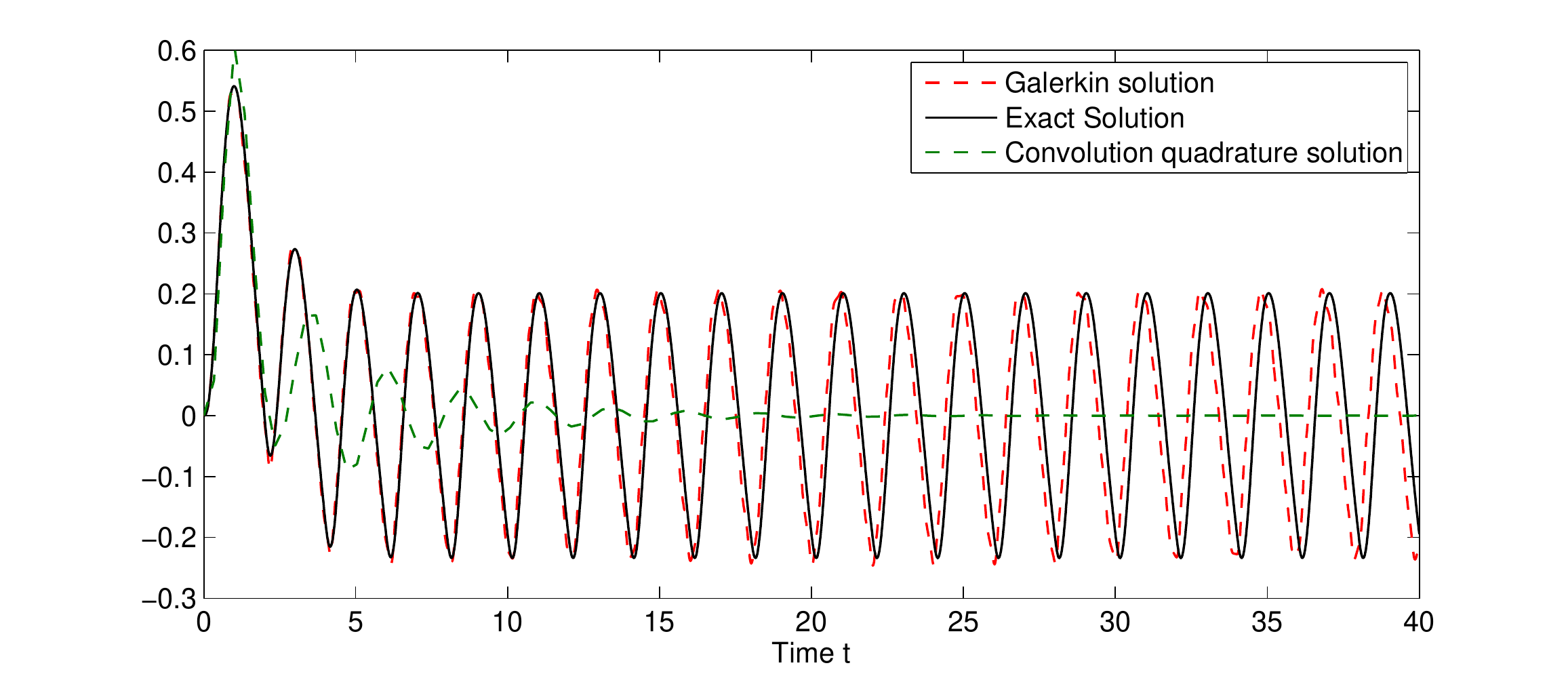}\caption{Galerkin
and Convolution Quadrature solution compared to the exact solution of
\eqref{BIE} for $\Gamma= \mathbb{S}^{2}$ and $g(x,t)=t^{4}\operatorname{e}%
^{-2t}$ in the time interval $[0,40]$.}%
\label{long_time}%
\end{figure}

It can be observed that the space-time Galerkin method leads to stable
solutions also for long time computations. Due to the energy preservation of
the method no numerical damping can be observed which is, e.g., typically the
case for time discretizations schemes based on convolution quadrature (cf.
Fig. \ref{long_time}). The slight shift of the numerical solution that is
present in Figure \ref{long_time} compared to the exact solution for large
times is due to the insufficient approximation in space and furthermore due to
the surface approximation of the sphere by flat triangles.\vspace
{\baselineskip}

\emph{A non-convex scatterer}\vspace{\baselineskip}

In Figure \ref{Torus} we consider the scattering of a Gaussian pulse from a
torus. We set the incoming wave as
\[
u_{\text{inc}}(x,t) := 8\cos(t-x_{1})\operatorname{e}^{-1.5(t-x_{1}-5)^{2}}
\quad\text{for}\quad(x,t)\in\mathbb{R}^{3}\times[0,12]
\]
and set the right hand side of the scattering problem \eqref{BIE} to
\[
g(x,t) = -u_{\text{inc}}(x,t)\quad\text{for}\quad(x,t)\in\Gamma\times[0,12].
\]
As illustrated in Figure \ref{Torus} the incoming wave travels in $x_{1}%
$-direction towards the torus. We discretize the torus with 1152 flat
triangles and use piecewise linear polynomials for the approximation in space.
For the temporal discretization we use 100 equidistant timesteps in the
interval [0,12] and approximate with local polynomial approximations spaces in
time of degree 1.

\begin{figure}[ht]
\centering
\includegraphics[width=1.0\textwidth]{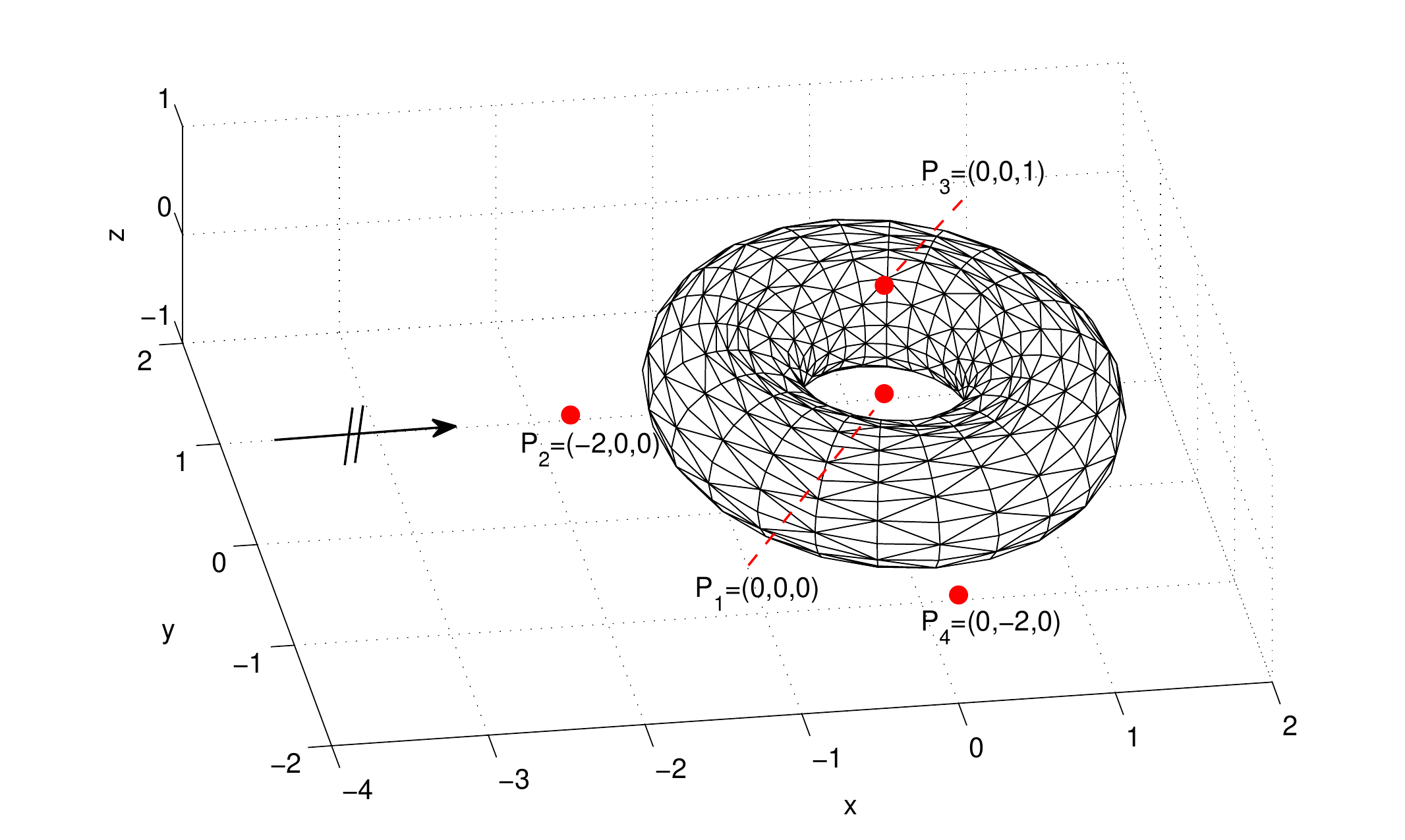}\caption{Scattering of a
Gaussian pulse from a torus with observation points $P_{1},\cdots, P_{4}$.}%
\label{Torus}%
\end{figure}

We compute the scattered wave at four observation points $P_{1},\ldots, P_{4}$
in the exterior domain of the torus. The results are illustrated in Figure
\ref{Torus_sol}. As expected, the scattered wave at the points $P_{1}$ and
$P_{3} $ exhibits small oscillations even after the incoming wave has passed.
This is due to the non-convexity of scatterer and the associated waves that
are trapped in the hole of the torus.\vspace{\baselineskip}

\begin{figure}[ht]
\centering
\subfigure[Solution $u(x_0,t)$ of the scattering problem at $x_0 = P_1 = (0,0,0).$]{
\centering
\includegraphics[width=0.47\textwidth]{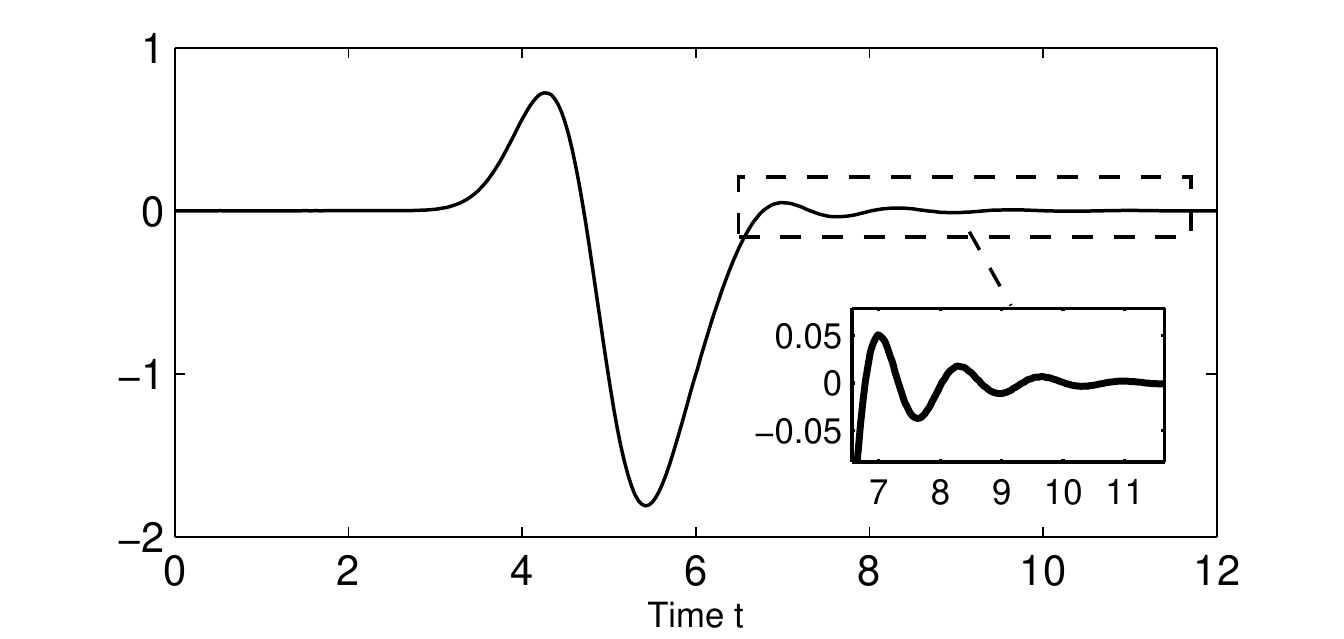}
} \hfil
\subfigure[Solution $u(x_0,t)$ of the scattering problem at $x_0 = P_3 = (0,0,1).$]{
\centering
\includegraphics[width=0.47\textwidth]{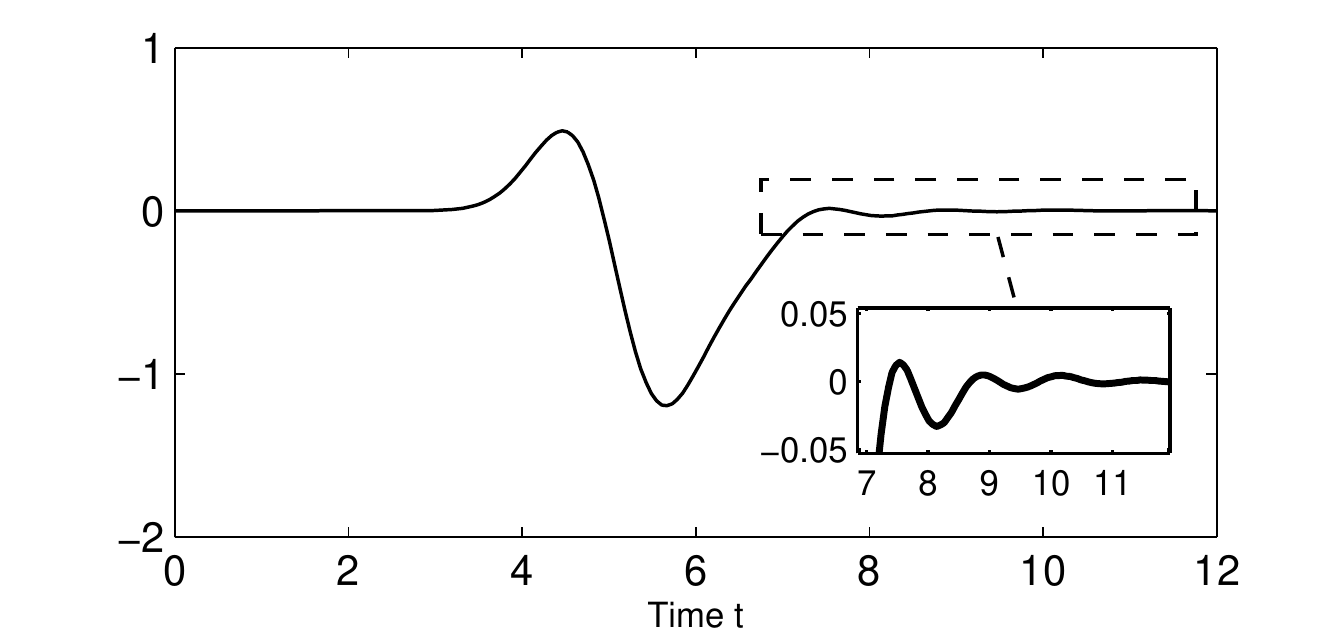}
}\newline%
\subfigure[Solution $u(x_0,t)$ of the scattering problem at $x_0 = P_2 = (-2,0,0).$]{
\centering
\includegraphics[width=0.47\textwidth]{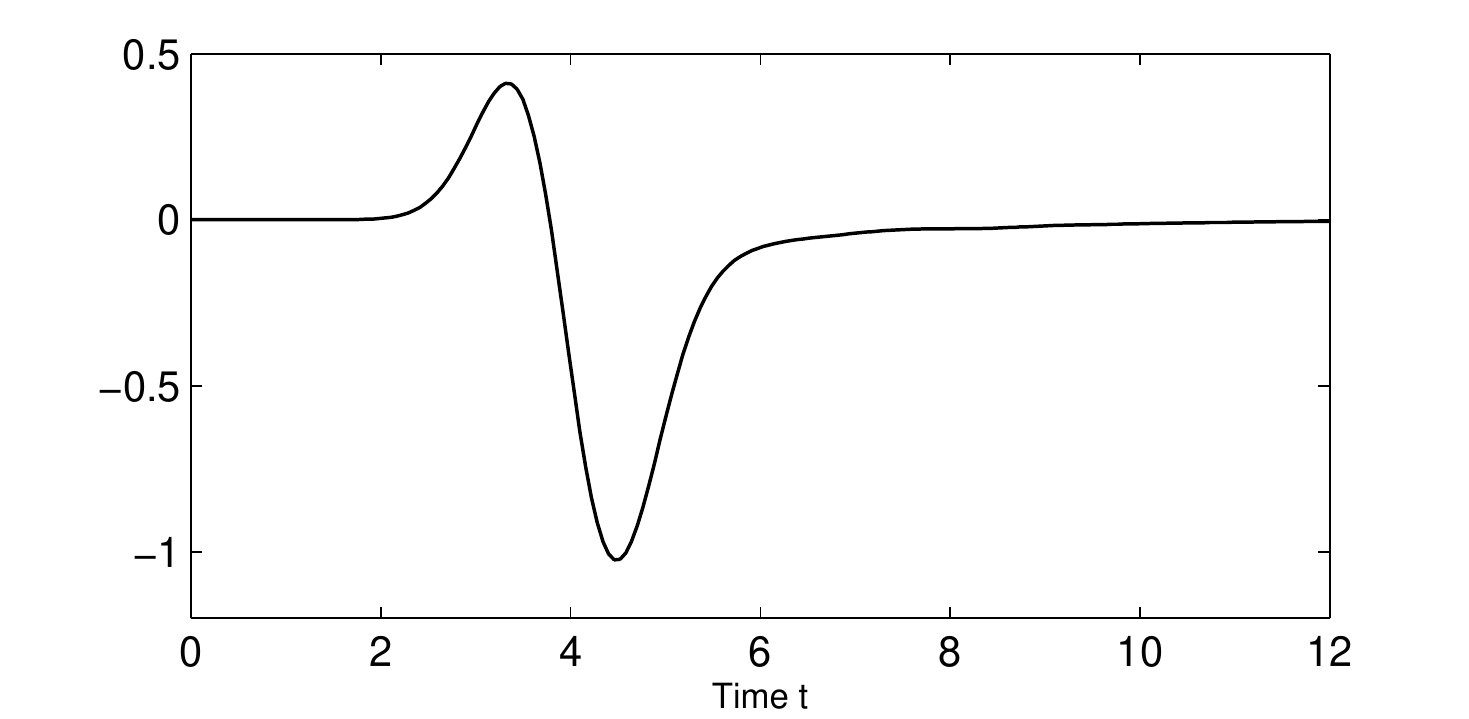}
}
\subfigure[Solution $u(x_0,t)$ of the scattering problem at $x_0 = P_4 = (0,-2,0).$]{
\centering
\includegraphics[width=0.47\textwidth]{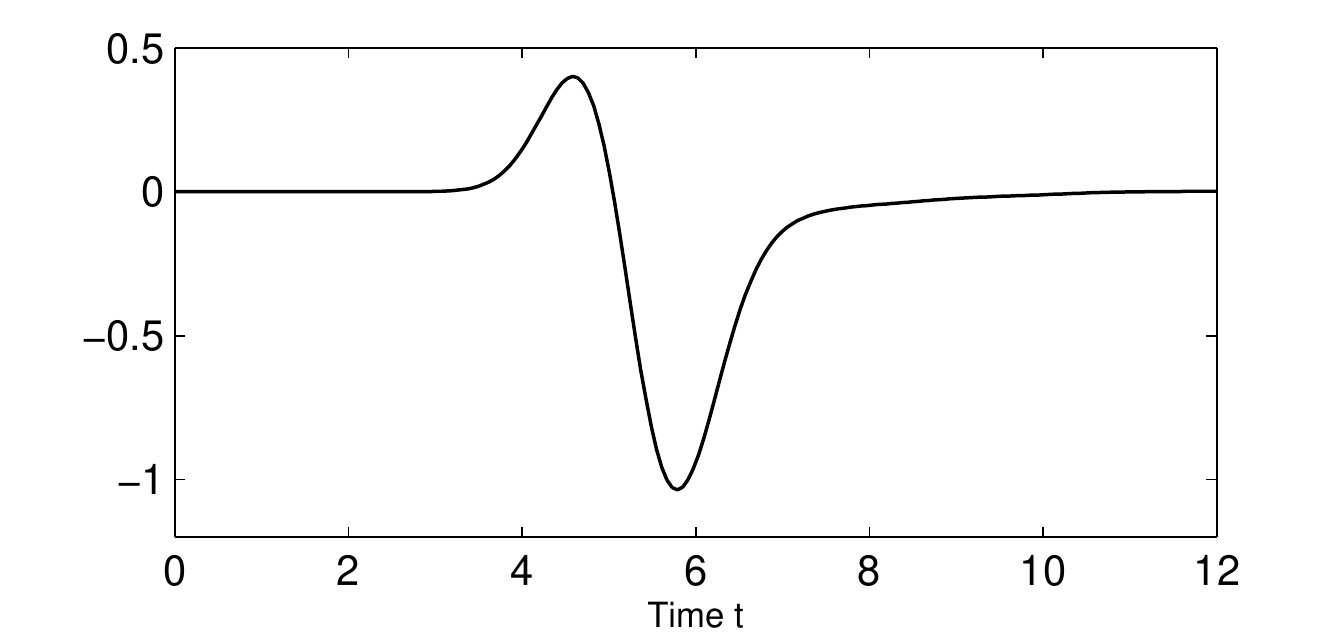}
}\caption{Solutions of the scattering problem from the torus in Figure
\ref{Torus} for the points $P_{1},\cdots, P_{4}$ in the exterior domain. }%
\label{Torus_sol}%
\end{figure}

\emph{Adaptivity in time}\vspace{\baselineskip}

In this subsection we present numerical experiments that show the performance
of the adaptive strategy described in Section \ref{Sec:ErrEst}. First we adopt
again the setting of a spherical scatterer $\Gamma=\mathbb{S}^{2}$ and a right
hand side of the form $g(x,t)=g(t)Y_{n}^{m}$. In this case the boundary
integral equation \eqref{BIE} decouples and leads to the purely time-dependent
problem: Find $\phi(t)$ such that
\begin{equation}
\int_{0}^{t}\mathcal{L}^{-1}(\lambda_{n})(\tau)\phi(t-\tau)d\tau
=g(t),\,\,\,t\in\lbrack0,T],\label{ATDRP:1dimProblem}%
\end{equation}
where $\mathcal{L}^{-1}$ denotes the inverse Laplace transform and
$\lambda_{n}(s)=I_{n+\frac{1}{2}}(s)K_{n+\frac{1}{2}}(s)$, where $I_{\kappa}$
and $K_{\kappa}$ are modified Bessel functions (cf. \cite{Sauter5} for
details). Note that $\phi(t)Y_{n}^{m}$, where $\phi(t)$ satisfies
\eqref{ATDRP:1dimProblem}, is a solution of the full problem \eqref{BIE}. It
is convenient to observe the behavior of the time-adaptive scheme (i.e. the
refinement process) using this one-dimensional problem since no spatial
discretization takes place that might have an influence on the results. In the
following we solve \eqref{ATDRP:1dimProblem} by a Galerkin method for two
different right-hand sides.\newline In a first experiment we set $n=0$ and
consider $g(t)=t^{1.5}\operatorname{e}^{-t}$ on the time interval $[0,1]$. The
exact solution of this problem is illustrated in Figure \ref{Fig:sol1}(a).
Since the solution involves the first derivative of $g$ (cf. \eqref{exactSol})
it is nonsmooth at $t=0$. For the numerical solution of this problem we use
local polynomial approximation spaces of degree $p=1$ and use the error
measure of Section \ref{Sec:Norm}. Figure \ref{Fig:sol1}(b) shows the error of
the adaptive scheme compared to the approximation using equidistant time
steps. Due to the nonsmoothness of the solution the equidistant approximation
converges only with suboptimal rate. The adaptive algorithm converges
significantly faster due to the successive refinement of the time grid towards
the origin.\newline In the second experiment we again set $n=0$ and consider
the right-hand side $g(t)=-\sin(10t)t^{3}\operatorname{e}^{-48(t-1)^{2}}$ on
the time interval $[0,4]$. The exact solution of this problem is depicted in
Figure \ref{Fig:sol2}(a). In this case the solution is smooth but oscillatory
around $t=1$ and $t=3$. In Figure \ref{Refinement} different refinement levels
of the adaptive approximation are shown. We start with a coarse time grid
consisting of only 4 time steps and iterate the adaptive procedure ten times.
It can be seen that at first the bump around $t=1$ is refined and only
afterwards the refinement around $t=3$ begins. Intuitively this seems to be
the right behavior since we solve a time-dependent wave propagation problem.
Thus the solution at a later time can only be accurately resolved if the
solution is already sufficiently approximated at earlier times. This behavior
of the adaptive scheme repeats for higher refinement levels as indicated by
the time grids at levels 8,9 and 10. The errors of the adaptive and the
equidistant approximation are depicted in Figure \ref{Fig:sol2}(b).\vspace
{0pt}

\begin{figure}[ht]
\centering
\subfigure[Exact solution]{
\centering
\includegraphics[width=0.47\textwidth]{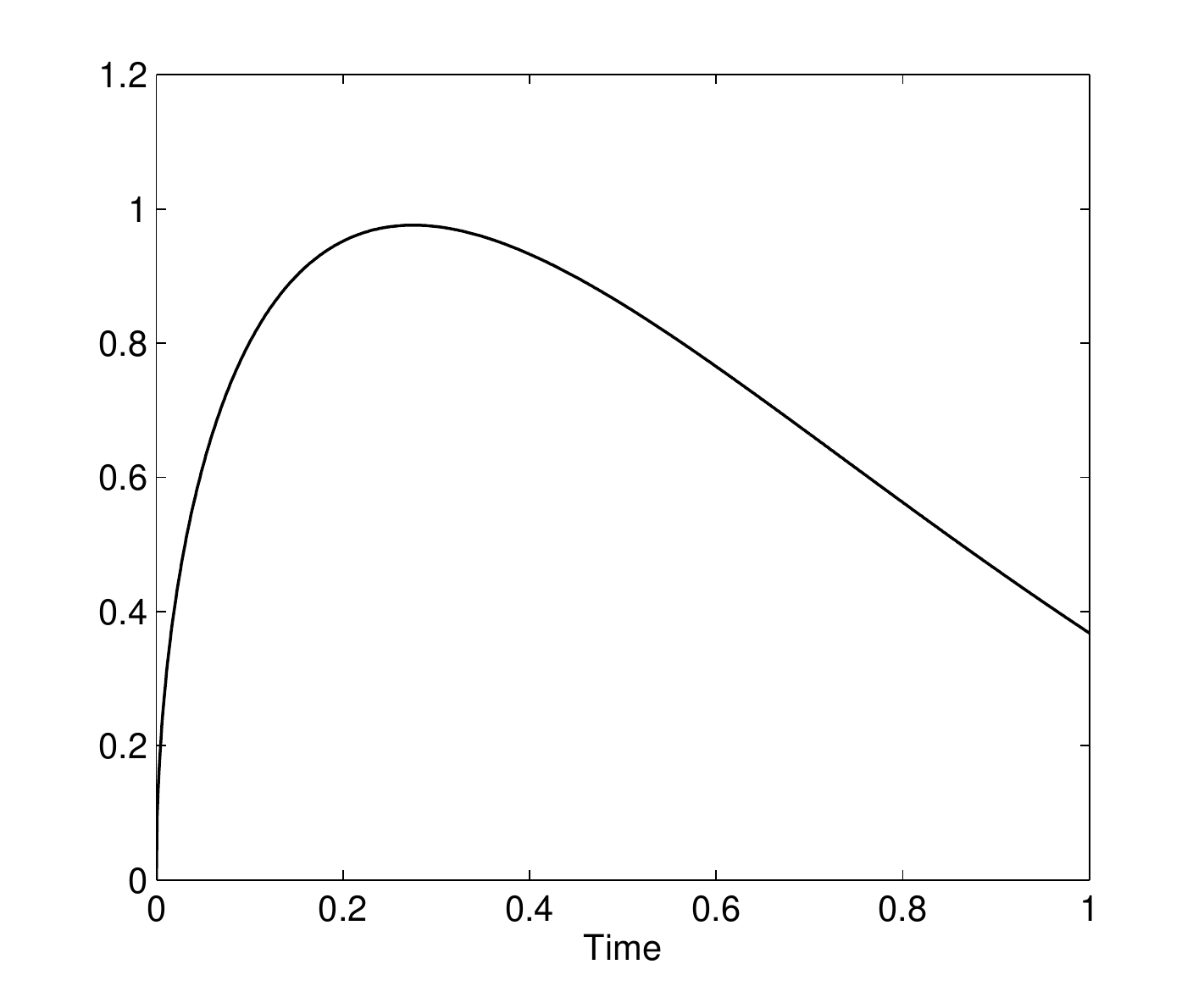}
} \hfil
\subfigure[Corresponding errors]{
\centering
\includegraphics[width=0.47\textwidth]{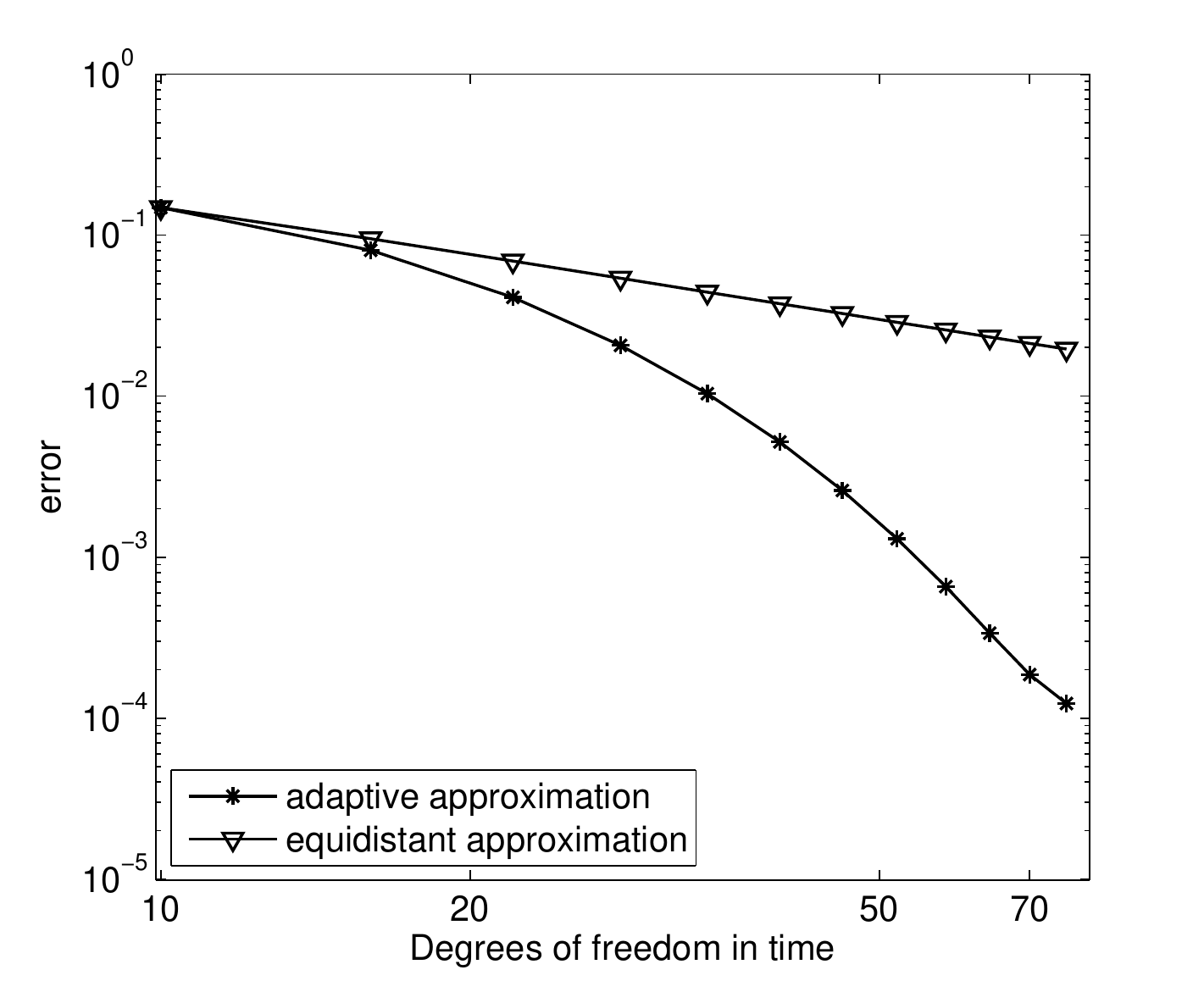}
}\caption{Solution $\phi(t)$ of problem \eqref{ATDRP:1dimProblem} for $g(t) =
t^{1.5}\operatorname{e}^{-t}$ and the corresponding errors of the adaptive and
the equidistant approximation. }%
\label{Fig:sol1}%
\end{figure}

\begin{figure}[ht]
\centering
\subfigure[Exact solution]{
\centering
\includegraphics[width=0.47\textwidth]{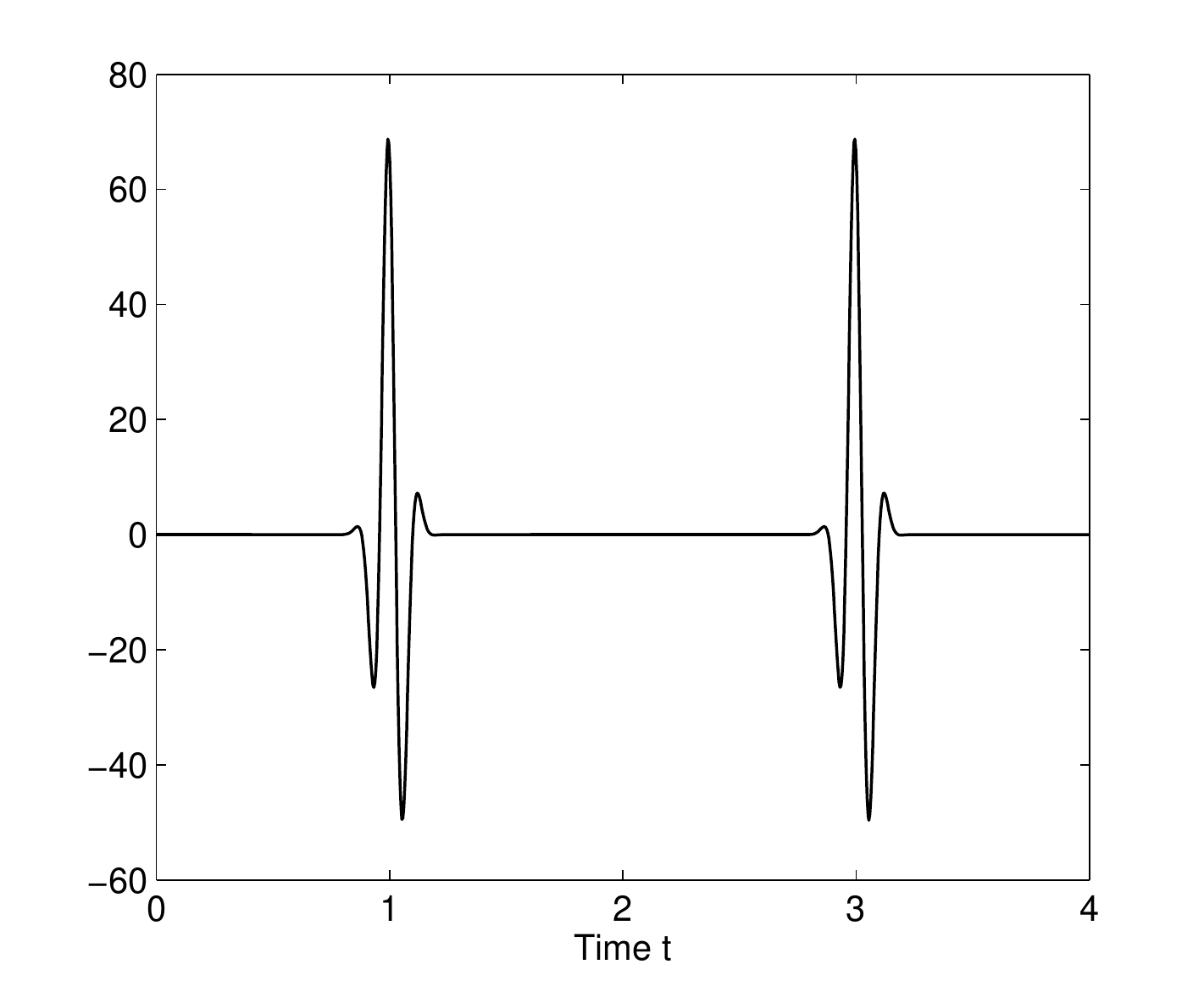}
} \hfil
\subfigure[Corresponding errors]{
\centering
\includegraphics[width=0.47\textwidth]{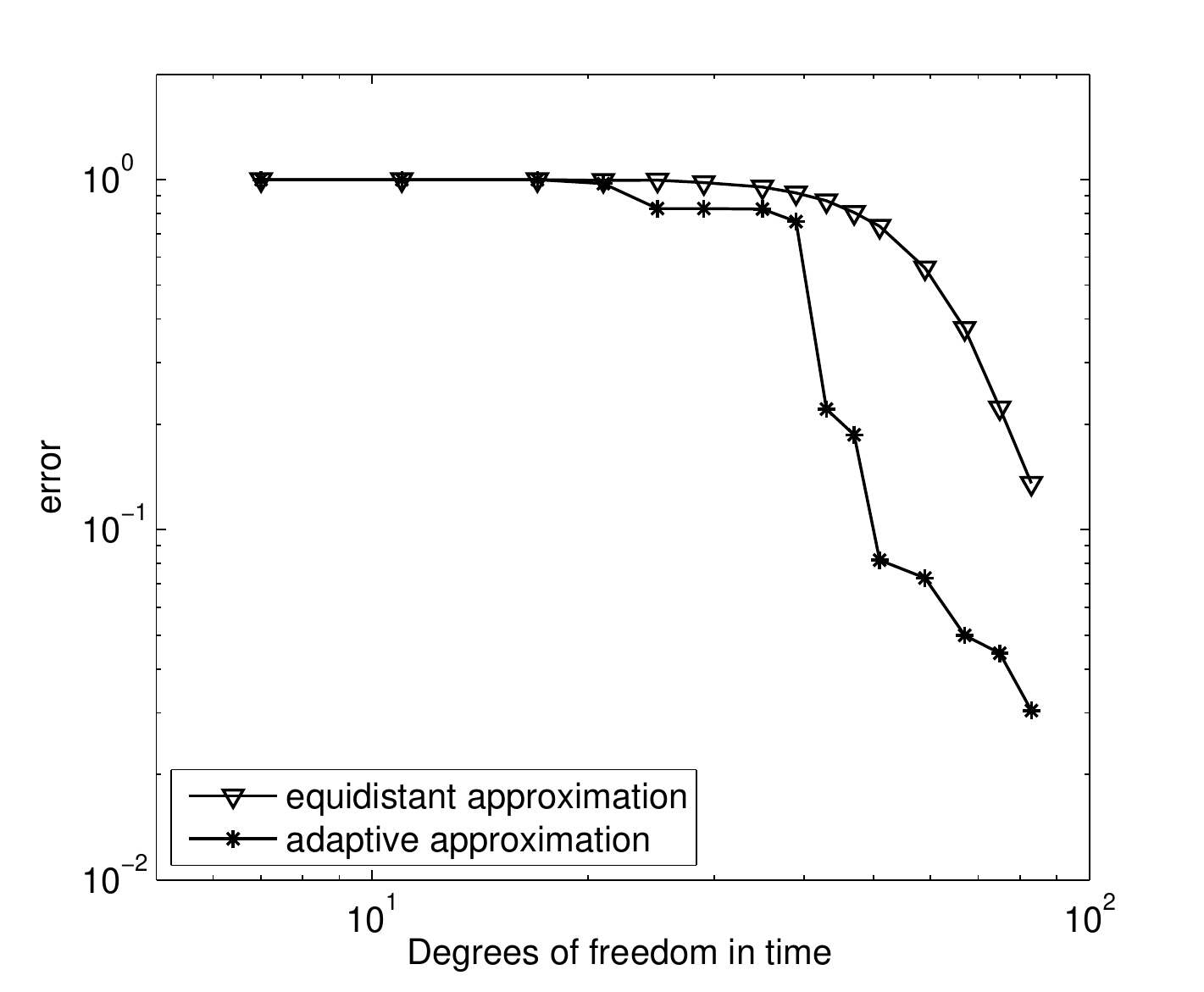}
}\caption{Solution $\phi(t)$ of problem \eqref{ATDRP:1dimProblem} for $g(t) =
-\sin(10t)t^{3}\operatorname{e}^{-48(t-1)^{2}}$ and the corresponding errors
of the adaptive and the equidistant approximation. }%
\label{Fig:sol2}%
\end{figure}

\begin{figure}[th]
\centering
\includegraphics[width=\textwidth]{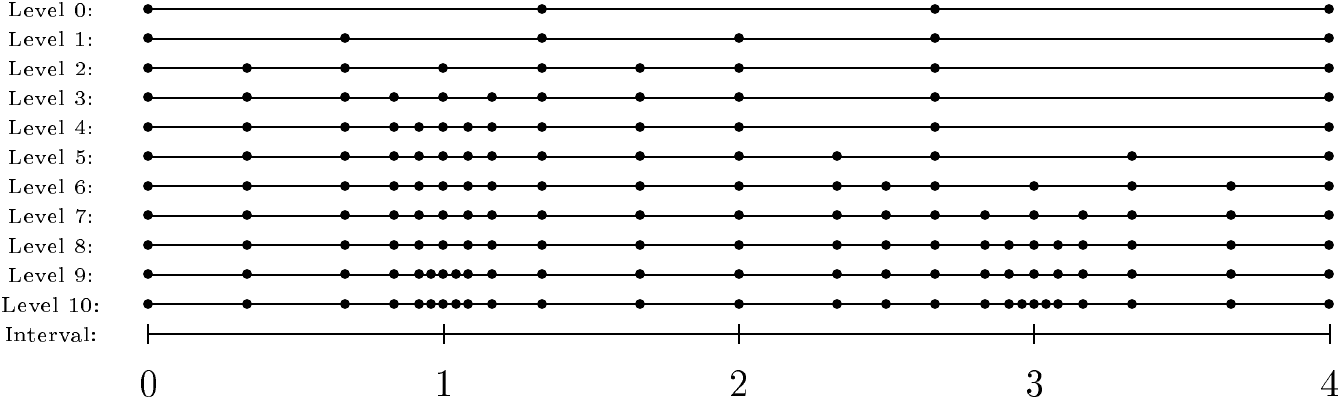}\caption{Different
refinement levels}%
\label{Refinement}%
\end{figure}

At last we test the adaptive algorithm for a full three-dimensional problem.
We use a spherical scatterer discretized into 616 triangles and we set
\begin{equation}
\label{rhs3D}g(x,t) = -H(t-x_{1}-2)\frac{(t-x_{1}-2)^{1.5}}{(t-x_{1}-2)^{2}+5}%
\end{equation}
for $x\in\mathbb{S}^{2}$ and $t\in[0,25]$. $H(\cdot)$ denotes the Heaviside
step function. This right-hand side corresponds again to an incoming wave
traveling in $x_{1}$-direction towards the scatterer which is met at $t=1$.
Due to the low regularity of the right-hand side we expect also low regularity
of the solution of the corresponding boundary integral equation. In Figure
\ref{Adap3d} two approximations of $\phi(x,t)$ at $(-1,0,0)^{\operatorname{T}%
}$ and $(1,0,0)^{\operatorname{T}}$ are illustrated . In both cases the
approximations were computed using local polynomial approximation spaces of
degree $p=1$ in time and piecewise linear functions in space. The solid lines
represent the numerical solution that was obtained using the time-adaptive
scheme. We started the adaptive algorithm with the coarse time grid $\{5\cdot
l,l=0,\ldots,5 \}$ and used the observation points $\Xi= \left\lbrace
(-1,0,0)^{\operatorname{T}}, (0,1,0)^{\operatorname{T}}%
,(1,0,0)^{\operatorname{T}} \right\rbrace $ for the refinement indicators. The
time grid after 6 iterations is shown in Figure \ref{Adap3d}. The dashed lines
represent the numerical solution that was obtained using an equidistant time
grid with the same number of timesteps. \begin{figure}[th]
\centering
\includegraphics[width=1.0\textwidth]{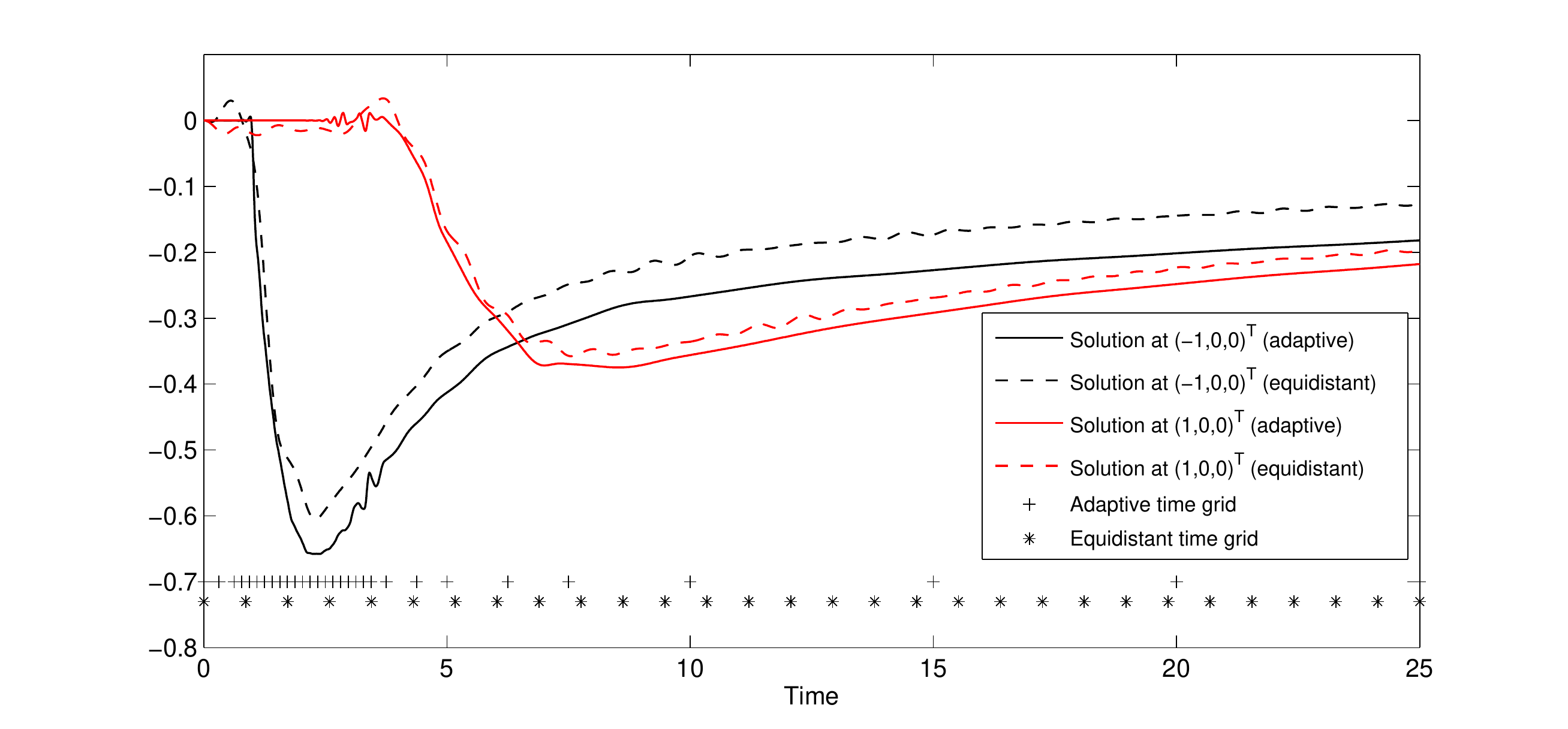}\caption{Comparison of the
equidistant and adaptive approximation for problem \eqref{rhs3D}. }%
\label{Adap3d}%
\end{figure}

The adaptive time grid is especially refined in the time interval $[1,3]$. The
nonsmoothness of the solution in this interval is not surprising since the
nonsmooth part of the incoming wave propagates through the obstacle at these
times. Due to the refined time grid the adaptive solution at
$(-1,0,0)^{\operatorname{T}}$ nicely captures the nonsmooth behavior of the
solution in this time interval. The insufficient accuracy of the equidistant
approximation in $[1,3]$ leads to a considerable shift of the numerical
solution at later times that cannot be corrected with additional timesteps
there. Similar observations can be made for the solution at
$(1,0,0)^{\operatorname{T}}$.\newline Once the nonsmoothness of the right-hand
side has passed the scatterer the solution seems considerably more smooth and
large time steps are sufficient for an accurate approximation.

\section{Conclusions\label{SecConcl}}

In this paper, we have introduced a fully discrete space-time Galerkin method
for solving the retarded potential integral equations. The focus was on the
efficient approximation of the integrals for building the system matrix, in
particular, for $C^{\infty}$ temporal basis functions and
combinations/convolutions thereof. It turned out that Gauss quadrature -- in
combination with regularizing coordinates for singular integrands -- converges
nearly as fast as for analytic functions in the accuracy regime of interest
$\left[  10^{-1},10^{-8}\right]  $.

In addition we have introduced an a posteriori error estimator for retarded
potential integral equations which is also employed for driving the adaptive
refinement of the time mesh. Numerical experiments show that the resulting
local error indicator captures very well local irregularities and oscillations
in the solution and the resulting time meshes are much more efficient compared
to uniform mesh refinement.\\
The adaptive refinement of the time mesh that we introduced in this paper is an important intermediate step towards a full space-time adaptive scheme. This will be an important further develpment in order to obtain a competitive method (see Remark \ref{rem_longTerm}).\\
Future work should furthermore address application to the Maxwell system and the theoretical analysis of the error estimator.\vspace{\baselineskip}

{\bf Acknowledgment.}  The second author gratefully acknowledges the support given by the Swiss National Science Foundation (No. P2ZHP2\_148705).

\bibliographystyle{abbrv}
\bibliography{mybib}

\end{document}